\documentclass[12pt]{amsart}
\usepackage[latin9]{inputenc}
\usepackage{hyperref}
\usepackage{color}
\usepackage{verbatim}
\usepackage{amsthm}
\usepackage{amssymb}
\usepackage{multicol}
\usepackage{tabularx}
\usepackage{graphicx}

\usepackage{pgf,tikz}
\usetikzlibrary{arrows}

\makeatletter

\newtheorem{case}{Case}
\newtheorem{subcase}{Subcase}[case]
\newtheorem{subsubcase}{Subcase}[subcase]

\def\qed{\hfill
\ifhmode\unskip\nobreak\fi\quad\ifmmode\Box\else$\Box$\fi\\ }
\newtheorem{theorem}{Theorem}

\newtheorem{lemma}[theorem]{Lemma}

\newtheorem*{claim}{Claim}

\numberwithin{equation}{section}
\numberwithin{figure}{section}
\usepackage{enumitem}		

\usepackage{fullpage}
\newtheorem{thm}[theorem]{Theorem}
\newtheorem*{thm*}{Theorem}
\newtheorem*{conj*}{Conjecture}

\usepackage{tikz}
\usepackage{pgf}
\newcount\mycount

\usepackage{caption}
\usepackage[labelformat=simple,labelfont={}]{subcaption}

\usetikzlibrary{shapes}
\usepackage{graphicx}

\makeatother

    \providecommand{\questionname}{Question}

\newcommand{\sU}{\mathcal{U}}
\newcommand{\obeta}{\bar{\beta}}
\newcommand{\oalpha}{\bar{\alpha}}
\newcommand{\kchi}{\chi'_{k\text{-int}}}
\newcommand{\tchi}{\chi'_{2\text{-int}}}

\begin{document}



\title{$k$-intersection edge-coloring subcubic planar multigraphs}

\author{M. Santana}
\thanks{Department of Mathematics, University of Illinois, Urbana, IL, 61801,
USA. This author's research is supported  by the NSF grant DMS-1266016 ``AGEP-GRS''.\\
E-mail address: \texttt{santana@illinois.edu.}}

\begin{abstract}
Given an edge-coloring of a simple graph, assign to every vertex $v$ a set $S_v$ comprised of the colors used on the edges incident to $v$.  The $k$-intersection chromatic index of a graph is the minimum $t$ such that the edge set can be properly $t$-colored, additionally requiring that for every two adjacent vertices $u$ and $v$, $|S_u \cap S_v| \le k$.  For all $k \neq 2$, this value is known for subcubic planar graphs, and furthermore, these values are best possible.  We naturally extend this definition to multigraphs with bounded edge multiplicity, and we show that every subcubic planar multigraph with edge multiplicity at most two has 2-intersection chromatic index at most 5, which is sharp.
\end{abstract}

\maketitle
\noindent{\small{Mathematics Subject Classification: 05C15 (05C10)}}{\small \par}

\noindent{\small{Keywords: $k$-intersection edge-coloring, subcubic graphs,  planar graphs.}}{\small \par}


\section{Introduction}

All multigraphs in this paper are loopless.  A \emph{proper edge-coloring} of a multigraph is an edge-coloring in which the edges of each color class form a matching in the original multigraph.  A \emph{strong edge-coloring}, introduced by Foquet and Jolivet (see \cite{FJ1, FJ2}),  is a proper edge-coloring in which we require the edges of each color class form an induced matching in the original multigraph.  The (\emph{strong}) \emph{chromatic index} of a multigraph $G$ is the minimum $t$ for which $G$ has a (strong) edge-coloring using $t$ colors.  

Given an edge-coloring of a simple graph, assign to every vertex $v$ a set $S_v$ comprised of the colors used on the edges incident to $v$.  For a fixed positive integer $k$, a $k$\emph{-intersection edge-coloring} is a proper edge-coloring in which $|S_u \cap S_v| \le k$ for all adjacent vertices $u$ and $v$.  The $k$\emph{-intersection chromatic index} of a simple graph $G$, denoted by $\kchi(G)$, is the minimum $t$ for which $G$ has a $k$-intersection edge-coloring using $t$ colors.

The notion of $k$-intersection edge-colorings was introduced in 2002 by Muthu, Narayanan, and Subramanian \cite{MNS} and was defined as above for simple graphs.  This same definition extends to loopless multigraphs in the natural way, however we require the edge multiplicity of the multigraph to be at most $k$, as otherwise it is not well-defined.  In particular, a 1-intersection edge-coloring exists only for simple graphs.

For simple graphs, when $k$ is at least the maximum degree, a $k$-intersection edge-coloring is equivalent to a proper edge-coloring, and furthermore, a 1-intersection edge-coloring is equivalent to strong edge-coloring.  Thus, the concept of $k$-intersection edge-coloring provides a sequence of parameters which join the notions of proper edge-colorings and strong edge-colorings.

Recently, Borozan et al \cite{Betal} show that computing $\kchi(G)$ is NP-complete for every $k \ge 1$.  They also compute bounds on $\kchi(G)$ for various families of graphs.  In particular, they show that every \emph{subcubic} graph $G$ (i.e., $G$ has maximum degree at most three) has $\tchi(G) \le 6$.  However, it is unknown whether or not this is best possible.  In this paper, we will restrict ourselves to subcubic planar graphs.

By Vizing's Theorem, every subcubic planar graph $G$  has $\kchi(G) \le 4$ for $k \ge 3$, and this is best possible.  Recently, Kostochka et al \cite{KLRSWY} show that the strong chromatic index of subcubic planar multigraphs is at most 9.  This is best possible and verifies a conjecture of Faudree et al \cite{FGST}.  As a corollary, every subcubic planar graph $G$ has $\chi'_{1\text{-int}}(G) \le 9$.   The aim of this paper is to complete the spectrum of $k$-intersection edge-colorings for subcubic planar graphs, by proving the following stronger statement.

\begin{thm}\label{thm:conj}
Every subcubic planar multigraph $G$ with edge multiplicity at most two has $\tchi(G) \le 5$.
\end{thm}

This is best possible by considering the complete graph on four vertices.  As mentioned previously,  2-intersection edge-coloring  a multigraph with edge multiplicity at least three is not well-defined.

The structure and proof of Theorem \ref{thm:conj} will follow very closely to that of Kostochka et al  \cite{KLRSWY}.  In Section \ref{sec:prelim}, we provide the notation we will use.  The remaining sections assume the existence of a minimal counterexample.  Section \ref{struct1} contains basic properties of a minimal counterexample, including the fact that it has no cycles of length three or four.  Additionally, if two vertices of degree 2 exist on a face, then the distance between them on the boundary is at least five.  The lemmas in Section \ref{struct2} will show that if a face has a vertex of degree 2 on its boundary, then the face has length at least eight.  Section \ref{struct3} contains two lemmas showing that every face of length five is surrounded by faces of length at least seven.  The proofs of these two lemmas are detailed and involve many pages of case analysis.  For the sake of brevity, these details can be found in the Appendix.  Lastly, Section \ref{sec:proof} contains a discharging proof based on the lemmas presented in Sections \ref{struct1}, \ref{struct2} and \ref{struct3}.


\section{Preliminaries and notation}\label{sec:prelim}

In the proof of Theorem \ref{thm:conj}, we will often remove vertices or edges from a minimal counterexample and obtain a 2-intersection edge-coloring of the remaining multigraph.  To aid us, we introduce some notation that we will use in our explanations.  As mentioned, much of this notation is also found in \cite{KLRSWY}.

Lower case Greek letters, such as $\alpha,\beta, \gamma, \delta$, will be used to denote arbitrary colors, and we use $\phi, \sigma, \psi$ to denote colorings.  Also an $i$\emph{-vertex} is a vertex of degree $i$ in our multigraph, and a $j$\emph{-face} is a face of length $j$ in our plane multigraph.  An $i^+$\emph{-vertex} and $j^+$\emph{-face} is a vertex of degree at least $i$ and a face of length at least $j$, respectively.

A coloring of a multigraph $G$ is \emph{good}, if it is a 2-intersection edge-coloring of $G$ using at most 5 colors.  
A \emph{partial} coloring of a multigraph $G$ is a coloring of any subset of $E(G)$.  Let $\phi$ be a partial coloring of a multigraph $G$.  For $v \in V(G)$, let $\sU_\phi(v)$ denote the set of colors used on the edges incident to $v$.  In partcular,  if no edges incident to $v$ are colored by $\phi$, then $\sU_\phi(v) = \emptyset$.   We say that a partial coloring $\phi$ is a \emph{good partial coloring} of $G$, if for any two adjacent vertices $v_1$ and $v_2$ in $G$, $|\sU_\phi(v_1) \cap \sU_\phi(v_2)| \le 2$.   At times we will refer to only one partial coloring which will not be named.  In these cases we will suppress the subscripts in the above notation.

Suppose $x_0x_1 \dots x_{k-1}$ is a cycle of length $k$ whose edges are all uncolored, and let $\alpha_0, \alpha_1$, $\dots$, $\alpha_{k-1}$ be colors.  We will say that we \emph{color the cycle in order with} $\alpha_0, \alpha_1, \dots, \alpha_{k-1}$, when we color $x_ix_{i+1}$ with $\alpha_i$, where $i$ is taken modulo $k$.


\section{Basic Properties}\label{struct1}

Everywhere below we assume $G$ to be a subcubic planar multigraph with edge multiplicity at most two contradicting Theorem \ref{thm:conj}.  Among all such counterexamples, we assume that $G$ has fewest vertices, and over all such counterexamples, has fewest edges.  $G$ is connected, as otherwise we can color each component by the minimality of $G$, and so obtain a good coloring of $G$.  As $G$ is planar, we assume $G$ to be a \emph{plane} multigraph in all the following statements.  That is, we consider $G$ together with an embedding of $G$ into the plane.

In this section, we will show several properties of $G$, including that $G$ is simple, has no small cycles and the distance between any two 2-vertices is at least four in $G$.  We will end this section by showing that 2-vertices are on the boundary of the same face, the distance between them on the boudary is at least five.

\begin{lemma}\label{muledges}
$G$ has no multiple edges, i.e., $G$ is a simple graph.
\end{lemma}
\begin{proof}
Suppose that $e_1$ and $e_2$ are parallel edges in $G$ with endpoints, $u$ and $v$.  Suppose first that $u$ and $v$ have a common neighbor $w$.  By the minimality of $G$, $G - e_1 - e_2$ has a good coloring.  Without loss of generality $\sU(w) \subseteq \{1,2,3\}$.    We then color $e_1$ and $e_2$ with 4 and 5, respectively.  This is a good coloring of $G$.

Now suppose $u$ has a neighbor $u'$ such that $u' \notin N_G(v)$.  $G' = G - \{u\} + u'v$ is a subcubic planar multigraph with edge multiplicity at most two, which has a good coloring by the minimality of $G$.  Suppose that $u'v$ is colored 1.  We then impose this coloring onto $G$ by coloring $uu'$ with 1.  If $v$ had another neighbor, say $v'$, then $vv'$ was colored with a color different from 1, say 2.  Our arim is to color the edges between $u$ and $v$ with colors from $\{3,4,5\}$.

Without loss of generality, suppose $3 \notin \sU(u')$.  If $\sU(v') \neq \{2,3,4\}$, color $e_1$ and $e_2$ with 3 and 4, otherwise color them with 3 and 5.  This yields a good coloring of $G$ and proves the lemma.
\end{proof}

Note that by Lemma \ref{muledges}, if $G'$ is obtained from $G$ by adding edges between distinct pairs of vertices, then $G'$ will always be a multigraph with edge multiplicity at most two.  We will use this in often to create auxiliary multigraphs smaller than $G$ that are subcubic planar, and have edge multiplicity at most two.

\begin{lemma}\label{extend}
Let $\phi$ be a good partial coloring of $G$, and let $uv \in E(G)$ be uncolored by $\phi$.  If either $\sU_\phi(u)$ or $\sU_\phi(v)$ is empty, then $\phi$ can be extended to another good partial coloring of $G$ by coloring $uv$.
\end{lemma}

\begin{proof}
Without loss of generality, suppose $\sU_\phi(u) = \emptyset$.  If $v$ is incident to at most one colored edge, then we simply color $uv$ properly, and we are done.  So we may assume that $v$ is a 3-vertex with other neighbors $v_1$ and $v_2$, and $\phi(vv_i) = i$ for $i \in \{1,2\}$.  If we cannot extend $\phi$ by coloring $uv$ with 3, then either $\sU_\phi(v_1)$ or $\sU_\phi(v_2)$ is $\{1,2,3\}$.  Similarly, if we cannot extend $\phi$ by coloring $uv$ with 4.  So we may assume $\sU_\phi(v_1) = \{1,2,3\}$ and $\sU_\phi(v_2) = \{1,2,4\}$.  Thus, we extend $\phi$ by coloring $uv$ with 5.
\end{proof}

\begin{lemma}\label{delta}
$G$ has minimum degree at least 2.
\end{lemma}

\begin{proof}
Suppose that $v$ is a $1$-vertex and $u$ is the neighbor of $v$. Then $G - v$ has a good coloring.  This is a good partial coloring of $G$ in which $\sU(v) = \emptyset$.  Thus, by Lemma \ref{extend}, $G$ has a good coloring.
\end{proof}

\begin{lemma}\label{cut-vertex}
$G$ has no cut-vertex and no cut-edge.
\end{lemma}
\begin{proof}
Since $G$ is subcubic, the existence of a cut-vertex implies the existence of a cut-edge.  Thus, it suffices to suppose that $G$ has a cut-edge, say $v_1v_2$.  For $i = 1,2$, let $H_i$ be the component of $v_1v_2$ containing $v_i$.  By Lemma \ref{delta}, $|V(H_i)| \ge 2$.  Define $G_1$ to be the graph consisting of $H_1$ together with $v_2$ and the edge $v_1v_2$.  Similarly define 
$G_2$ to be the graph consisting of $H_2$ together with $v_1$ and the edge $v_1v_2$. 

By the minimality of $G$, $G_1$ and $G_2$ have good colorings, $\phi_1$ and $\phi_2$, respectively.  We may assume $\sU_{\phi_1}(v_1) \subseteq \{1,2,3\}$,  $\sU_{\phi_2}(v_2) \subseteq \{1,4,5\}$ with $\phi_1(v_1v_2) = \phi_2(v_1v_2) = 1$.  Merging these two colorings yields a good coloring of $G$.
\end{proof}

\begin{lemma}\label{2-edge-cut} 
If $\{e_1, e_2\}$ is an edge-cut in $G$, then $e_1$ and  $e_2$ are adjacent to each other.
\end{lemma}
\begin{proof}
If not, then we have an edge-cut $\{u_{1}v_{1},u_{2}v_{2}\}$ in $G$ that is a matching.  We may assume that $u_1$ and $u_2$ are in the same component of $G - u_1v_1 -  u_2v_2$ so that we can define $H_u$ to be the component of $G - u_1v_1 - u_2v_2$ containing $v_1$ and $v_2$.  Let $H_v = G - H_u$, and let $G_u$ be the graph consisting of $H_u$ together with a new vertex $u$ whose neighborhood is $\{v_1,v_2\}$.  Similarly, let $G_v$ be the graph consisting of $H_v$ together with a new vertex $v$ whose neighborhood is $\{u_1,u_2\}$.  Observe that $G_u$ and $G_v$ are subcubic planar multigraphs, and so by the minimality of $G$, $G_u$ and $G_v$ have good colorings $\phi_u$ and $\phi_v$, respectively.  

Permute these colorings so that for $i \in \{1,2\}$, $\phi_v(u_iv) = \phi_u(v_iu)$.  If $\sU_{\phi_u}(v_1) \neq \sU_{\phi_v}(u_1)$ and $\sU_{\phi_u}(v_2) \neq \sU_{\phi_v}(v_2)$, then merging these two colorings yields a good coloring of $G$.  So suppose that $\sU_{\phi_u}(v_1) = \sU_{\phi_v}(u_1) = \{1,2,3\}$ with $\phi_u(uv_1) = \phi_v(vu_1) = 1$.  

Since $\phi_u$ is a good coloring, either $\phi_u(uv_2) \in \{2,3\}$ or $\phi_u(uv_2) \in \{4,5\}$.  Suppose first that $\phi_u(uv_2) = 2$.  Suppose $\sU_{\phi_v}(u_2) = \sU_{\phi_u}(v_2)$.  If $3 \in \sU_{\phi_u}(v_2)$, then switch 3 with a color from $\{4,5\}\setminus\sU_{\phi_u}(v_2)$ in $\phi_u$.  If $3 \notin \sU_{\phi_u}(v_2)$, then switch 3 with a color in $\{4,5\}\cap \sU_{\phi_u}(v_2)$.  In either case, merging this new coloring with $\phi_v$ yields a good coloring of $G$.

So $\sU_{\phi_v}(u_2) \neq \sU_{\phi_u}(v_2)$.  Suppose $\sU_{\phi_u}(v_2) = \{2,4,5\}$.  If $\{4,5\} \cap \sU_{\phi_v}(u_2) = \emptyset$, then switch 3 with either 4 or 5 in $\phi_u$.  Otherwise, switch 3 with a color in $\{4,5\} \cap \sU_{\phi_v}(u_2)$ in $\phi_u$.  Now suppose $\{4,5\} \cap \sU_{\phi_u}(v_2) = \{4\}$.  If $3 \in \sU_{\phi_u}(v_2)$, then switch 3 with 4 in $\phi_u$, otherwise switch 3 with 5 in $\phi_u$.  A similar argument holds when $\{4,5\} \cap \sU_{\phi_u}(v_2) = \{5\}$, so $\sU_{\phi_u}(v_2) = \{1,2,3\}$.  By symmetry, $\sU_{\phi_v}(u_2) = \{1,2,3\}$, however we assume $\sU_{\phi_v}(u_2) \neq \sU_{\phi_u}(v_2)$, a contradiction.   In all cases, we obtain either a contradiction or a new good coloring of $G_u$, $\phi'_u$ such that $\sU_{\phi'_u}(v_2) \neq \sU_{\phi_v}(u_2)$ and $\sU_{\phi'_u}(v_1) \neq \sU_{\phi_v}(u_1)$.  So merging $\phi_v$ with $\phi'_u$ yields a good coloring of $G$.

Thus, it remains to consider when $\phi_u(uv_2) = 4$.  Suppose $\sU_{\phi_u}(v_2) = \sU_{\phi_v}(u_2)$.  If $5 \in \sU_{\phi_u}(v_2)$, switch 5 with a color in $\{2,3\}\setminus\sU_{\phi_u}(v_2)$ in $\phi_u$.  Otherwise, switch 5 with a color in $\{2,3\} \cap \sU_{\phi_u}(v_2)$.

So $\sU_{\phi_u}(v_2) \neq \sU_{\phi_v}(u_2)$.  Suppose $\sU_{\phi_u}(v_2) = \{2,3,4\}$.  If $\{2,3\} \cap \sU_{\phi_v}(u_2) = \emptyset$, then switch 5 with either 2 or 3 in $\phi_u$.  Otherwise, switch 5 with a color in $\{2,3\} \cap \sU_{\phi_v}(u_2)$ in $\phi_u$.  Now suppose $\{2,3\} \cap \sU_{\phi_u}(v_2) = \{2\}$.  If $5 \in \sU_{\phi_u}(v_2)$, switch 5 with 2 in $\phi_u$.  Otherwise, switch 5 with 3 in $\phi_u$.  A similar argument holds when $\{2,3\} \cap \sU_{\phi_u}(v_2) = \{3\}$, so $\sU_{\phi_u}(v_2) = \{1,4,5\}$.  By symmetry, $\sU_{\phi_v}(u_2) = \{1,4,5\}$, however we assume $\sU_{\phi_v}(u_2) \neq \sU_{\phi_u}(v_2)$, a contradiction.  In all cases, we obtain either a contradiction or a new coloring of $G_u$, $\phi'_u$ such that $\sU_{\phi'_u}(v_2) \neq \sU_{\phi_v}(u_2)$ and $\sU_{\phi'_u}(v_1) \neq \sU_{\phi_v}(u_1)$.  So merging $\phi_v$ with $\phi'_u$ yields a good coloring of $G$.
\end{proof}

\begin{lemma}\label{NoTriangle}
$G$ has no triangles.
\end{lemma}
\begin{proof}
Suppose that $x,y,z$ is a triangle in $G$.  Suppose first that $x$ is a 2-vertex, and let $y_1$ and $z_1$ be the potential third neighbors of $y$ and $z$, respectively.  By the minimality of $G$, $G - x$ has a good coloring.  Suppose that $yz$ is colored 1 and $\sU(y) \cup \sU(z) \subseteq \{1,2,3\}$.  Without loss of generality $\sU(y) \subseteq \{1,2\}$.  If $1 \notin \sU(z')$, color $xz$ with either 4 or 5.  Otherwise, color $xz$ with a color from $\{4,5\}\setminus\sU(z')$.  Without loss of generality, suppose $xz$ is colored with 5.  If $1 \notin \sU(y')$, color $xy$ with either 3 or 4.  Otherwise, color $xz$ with a color from $\{3,4\}\setminus\sU(z')$.  This yields a good coloring of $G$ so that every vertex in  $\{x,y,z\}$ is a 3-vertex.

Let $x_1, y_1, z_1$ be the third neighbors of $x, y, z$, respectively.  If $x_1 = y_1 = z_1$, then $G = K_3$, and we are done.  So suppose $y_1 \neq z_1$.  $G' = G - \{x,y,z\} + y_1z_1$ is a subcubic planar multigraph with multiplicity at most two.  By the minimality of $G$, $G'$ has a good coloring.  We impose this coloring onto $G$ by coloring $yy_1$ and $zz_1$ with the color used on the added $y_1z_1$ in $G'$.  This yields a good partial coloring of $G$, call it $\phi$.  

Suppose $\phi(yy_1) = \phi(zz_1) = 1$.  We may assume that either both $y_1$ and $z_1$ are 2-vertices or that $y_1$ is a 3-vertex.  In either case, suppose $\sU_\phi(y_1) \subseteq \{1,2,3\}$.  By the existence of the added $y_1z_1$ in $G'$, $\sU_\phi(z_1) \neq \{1,2,3\}$.  Therefore, we can extend $\phi$ by coloring $xz$ and $yz$ from $\{2,3\}$ so that $xz$ is colored from $\{2,3\}\setminus\{\phi(xx_1)\}$.  Without loss of generality, assume $xz$ and $yz$ are colored with 2 and 3, respectively.

Suppose $\phi(xx_1) \in \{1,2,3\}$.  If $\phi(xz) \notin \sU_\phi(x_1)$, color $xy$ with either 4 or 5. Otherwise, color $xy$ with a color from  $\{4,5\}\setminus\sU_\phi(x_1)$. This yields is a good coloring of $G$.

So $\phi(xx_1) \in \{4,5\}$.  Without loss of generality, suppose $\phi(xx_1) = 4$.  If $\sU_\phi(x_1) \neq \{2,5\}$, color $xy$ with 5.  Otherwise, color $xy$ with 5, and then recolor $xz$ and $yz$ with 3 and 2, respectively.  In either case, we obtain a good coloring of $G$.
\end{proof}

\begin{lemma}\label{distance>=3}
The distance between any two 2-vertices is at least three.
\end{lemma}
\begin{proof}
Suppose first that $u$ and $v$ are adjacent 2-vertices in $G$.  By Lemma \ref{NoTriangle}, $u$ and $v$ have distinct neighbors $u_1$ and $v_1$, respectively.  By the minimality of $G$, $G - uv$ has a good coloring.  We then properly color $uv$ to obtain a good coloring of $G$.  

Now suppose that $v$ is a 3-vertex with neighbors $x,y, z$ such that $x$ and $y$ are 2-vertices.  If $z$ is a 2-vertex, then by the minimality of $G$, $G - v$ has a good coloring.  We then properly color $xv, yv, zv$ to obtain a good coloring of $G$.  So we may assume that $z$ is a 3-vertex.  

By the minimality of $G$, $G - xv - yv$ has a good coloring.  Suppose that $\sU(z) = \{1,2,3\}$, $\sU(x) = \{\alpha\}$, and $\sU(y) = \{\beta\}$.  We first color $xv$ with a color in $\{4,5\}\setminus\{\alpha\}$, and then color $yv$ properly.  This yields a good coloring of $G$.
\end{proof}

\begin{lemma}\label{separating}
$G$ has no separating cycle of length four or five.
\end{lemma}

\begin{proof}
We first show that $G$ has no 4-cycle with a 2-vertex.  Suppose that $x_0x_1x_2x_3$ is a 4-cycle.   Suppose $x_0$ is a 2-vertex.  By Lemma \ref{distance>=3}, $x_0$ is the only 2-vertex on this 4-cycle.  For $i \in \{1,2,3\}$, let $y_i$ be the third neighbor of $x_i$.  By Lemma \ref{NoTriangle}, the $y_i$'s are not on this 4-cycle, $y_1 \neq y_2,$ and $y_2 \neq y_3$.  

Let $G'$ be formed from $G$ be removing $x_0, x_1, x_2, x_3$, adding a new vertex $x$ and edges $xy_1, xy_2, xy_3$.  If $y_1 = y_3$, then $xy_1, xy_3$ are parallel edges.  Regardless, $G'$ is still a subcubic planar multigraph with multiplicity at most two.  Thus, by the minimality of $G$, $G'$ has a good coloring.  We impose this good coloring onto $G$ by coloring $x_iy_i$ with the color on $xx_i$ for $i \in \{1,2,3\}$, and  without loss of generality assume $x_iy_i$ is colored $i$.  

We extend this to another good partial coloring of $G$, call it $\phi$, by coloring $x_0x_1$ with 2 and $x_1x_2$ with 3.  if $3 \notin \sU_\phi(y_2)$, color $x_2x_3$ with either 4 or 5.  Otherwise, color $x_2x_3$ with a color from $\{4,5\}\setminus\sU_\phi(y_2)$.  Let $\alpha$ be the color used on $x_2x_3$.  If $\alpha \notin \sU_\phi(y_3)$, color $x_3x_0$ with a color from $\{1,4,5\}\setminus\{\alpha\}$.  Otherwise, color $x_3x_0$ with a color from $\{1,4,5\}\setminus(\sU_\phi(y_3) \cup \{\alpha\})$.  This yields a good coloring of $G$ so that $G$ has no 4-cycle with a 2-vertex.  We will use this to show that $G$ has no separating 4-cycle or 5-cycle.

If on the contrary, $G$ has a separating 4-cycle or 5-cycle, call it $C$.  By Lemma \ref{NoTriangle}, $C$ has no chords, and as $G$ is subcubic, each vertex of $C$ is incident to at most one edge not on $C$.   Since $\lfloor \frac{5}{2}\rfloor=  2$, by symmetry we may assume that there are at most two edges inside $C$ that are incident to vertices on $C$ (recall that $G$ is assumed to be embedded in the plane).  If there is exactly one such edge, then $G$ has a cut-edge, contradicting Lemma \ref{cut-vertex}.  So, we have two such edges, which are in fact cut-edges, and by Lemma \ref{2-edge-cut}, these edges  share a common endpoint, say $u$, inside of $C$.  Now, $u$ is a 2-vertex, as otherwise it would be a cut-vertex with a cut-edge.  However, $u$ together with the vertices of $C$ has either a triangle or a 4-cycle containing a 2-vertex, contradicting Lemma \ref{NoTriangle} or the above, respectively.  Thus, $G$ has no separating 4-cycle or 5-cycle.  
\end{proof}

\begin{lemma}\label{No4cycle}
$G$ has no 4-cycle.
\end{lemma}

\begin{proof}
Suppose that $x_0x_1x_2x_3$ is a 4-cycle in $G$.  By Lemma \ref{separating}, this cycle is a 4-face and as is shown in the proof of Lemma \ref{separating}, each $x_i$ is a 3-vertex.  As a result, we let $y_i$ denote the third neighbor of $x_i$.  By Lemmas \ref{NoTriangle} and \ref{separating}, the $y_i$'s are distinct and not on the 4-cycle.  Let $G' = G - \{x_0, x_1, x_2, x_3\} + y_0y_1 + y_2y_3$.  Observe that $G'$ is a subcubic planar multigraph with multigraph  with multiplicity at most two.  Thus, by the minimality of $G$, $G'$ has a good coloring.  We impose this good coloring onto $G$ by coloring $x_0y_0$ and $x_1y_1$ with the color assigned to $y_0y_1$, and coloring $x_2y_2$ and $x_3y_3$ with the color assigned to $y_2y_3$.  Let $\phi$ denote this good partial coloring of $G$.  

By symmetry, we may assume that either $y_0$ and $y_1$ are both 2-vertices, or that $y_0$ is a 3-vertex.  Now without loss of generality, assume $\sU_\phi(y_0) \subseteq \{1,2,3\}$ with $\phi(x_0y_0) = \phi(x_1y_1) = 1$.  Thus by the existence of $y_0y_1$ in $G'$, $\sU_\phi(y_1) \neq \{1,2,3\}$.  We proceed based on $\phi(x_2y_2)$, which up to relabeling, we may assume is in $\{1,2,4\}$.

\begin{case}
$\phi(x_2y_2) = \phi(x_3y_3) = 1$.
\end{case}

Suppose $\sU_\phi(y_3) \neq \{1,4,5\}$.  We color $x_0x_1$ and $x_1x_2$ with 2 and 3, respectively.   If $3 \notin \sU_\phi(y_2)$, we color $x_2x_3$ with either 4 or 5.  Otherwise, color $x_2x_3$ with a color from $\{4,5\}\setminus\sU_\phi(y_2)$.  In either case, let $\alpha$ denote the color used on $x_2x_3$, and we color $x_3x_0$ from $\{4,5\}\setminus\{\alpha\}$.  This yields a good coloring of $G$.

Thus, $\sU_\phi(y_3) = \{1,4,5\}$.  A similar argument holds when $\sU_\phi(y_1) \neq \{1,4,5\}$ by coloring $x_2x_3$ and $x_3x_0$ with 3 and 2, respectively.  So $\sU_\phi(y_1) = \{1,4,5\}$.  We now color $x_0x_1$ and $x_1x_2$ with 4 and 2, respectively.  If $2 \notin \sU_\phi(y_2)$, we color $x_2x_3$ with either 3 or 5.  Otherwise, we color $x_2x_3$ with a color from $\{3,5\}\setminus\sU_\phi(y_2)$.  In either case, let $\beta$ denote the color used on $x_2x_3$, and we color $x_3x_0$ from $\{3,5\}\setminus\{\beta\}$.  This yields a good coloring of $G$.

\begin{case}
$\phi(x_2y_2) = \phi(x_3y_3) = 2$.
\end{case}

The same argument as above yields $\sU_\phi(y_3) = \{2,4,5\}$.  By the existence of $y_2y_3$ in $G'$, $\sU_\phi(y_2) \neq \{2,4,5\}$.   Now $\sU_\phi(y_1) = \{1,4,5\}$, otherwise color the cycle in order with 5, 4, 5, 3.  We now color $x_2x_3$ and $x_3x_0$ with 1 and 4, respectively.  If $1 \notin \sU_\phi(y_2)$, we color $x_1x_2$ with either 3 or 5.  Otherwise, we color $x_1x_2$ with a color from $\{3,5\}\setminus\sU_\phi(y_2)$.  In either case, let $\alpha$ denote the color used on $x_1x_2$, and color $x_3x_0$ from $\{3,5\}\setminus\{\alpha\}$.  This yields a good coloring of $G$.

\begin{case}
$\phi(x_2y_2) = \phi(x_3y_3) = 4$.
\end{case}

We begin by coloring $x_0x_1, x_1x_2, x_3x_0$ with 2, 3, and 5, respectively.  Suppose $3 \notin \sU_\phi(y_2)$.  If $5 \notin \sU_\phi(y_3)$, color $x_2x_3$ with either 1 or 2.  Otherwise, color $x_2x_3$ with a color from $\{1,2\}\setminus\sU_\phi(y_3)$.  In either case, this yields a good coloring of $G$.  So $3 \in \sU_\phi(y_2)$, and by a similar argument, $5 \in \sU_\phi(y_3)$.

We now color $x_0x_1, x_1x_2, x_3x_0$ with 3, 2, and 5, respectively.  If $2 \notin \sU_\phi(y_2)$, color $x_2x_3$ with a color from $\{1,3\}\setminus\sU_\phi(y_3)$.  So $2 \in \sU_\phi(y_2)$ and $\sU_\phi(y_2) = \{2,3,4\}$.  Also $\sU_\phi(y_3) = \{1,4,5\}$, otherwise color the cycle in order with 3, 2, 1, 5.  This is a good coloring of $G$.

If $\sU_\phi(y_1) \neq \{1,4,5\}$, color the cycle in order with 4, 5, 3, 2, respectively.  Otherwise, color the cycle in order with 4, 2, 5, 3, respectively.  These yield a good colorings of $G$.

As we have exhausted all cases, this proves the lemma.
\end{proof}

In the proof of Lemma \ref{No4cycle}, we repeat the same argument several times, and we will continue to do so in much of the following.  Thus, for the sake of brevity we will replace this argument with a short statement.  Let $\phi$ be a good partial coloring of $G$, and let $xy$ be a colored edge in $G$ under $\phi$ such that $x$ is a 3-vertex incident to exactly one uncolored edge, call it $e_1$.  Let $\gamma$ be the color used on the colored edge incident to $x$ that is not $xy$, and let $\alpha$ and $\beta$ be colors such that $\{\alpha, \beta\} \cap \{\phi(xy), \gamma\} = \emptyset$.  Lastly, let $\{e_2,\dots, e_k\}$ and $\{e_1', \dots, e_{k'}'\}$ be two, possibly empty, collections of uncolored edges such that $\{e_2,\dots, e_k\} \cap \{e_1', \dots, e_{k'}'\} = \emptyset$.  We will say that we \emph{color} $e_1,e_2, \dots, e_k$ (\emph{and} $e_1', e_2', \dots, e_{k'}'$) \emph{from} $\{\alpha,\beta\}$ \emph{with respecto to} $\gamma$ \emph{and} $\sU_\phi(y)$ by coloring $e_1,e_2,\dots, e_k$ with $\alpha$ or $\beta$, and coloring $e_1', e_2',\dots, e_{k'}'$ with $\beta$ or $\alpha$, respectively.  They way $e_1$ is colored is as follows.  If $\gamma \notin \sU_\phi(y)$, we can color $e_1$ with either $\alpha$ or $\beta$.  Otherwise, we color $e_1$ with a color from $\{\alpha,\beta\}\setminus\sU_\phi(y)$.  By doing so and choosing $\alpha$ and $\beta$ carefully, this will extend $\phi$ to another good partial coloring of $G$.

\begin{lemma}\label{distance>=4}
The distance between any two 2-vertices is at least four.
\end{lemma}

\begin{proof}
Now suppose $x_0x_1x_2x_3x_4x_5$ is a path in $G$ where $x_1$ and $x_4$ are 2-vertices.  By Lemma \ref{distance>=3}, $x_0, x_2, x_3$, and $x_5$ are 3-vertices.  Let $y_2$ and $y_3$ be the third neighbors of $x_2$ and $x_3$, respectively.  By Lemmas \ref{NoTriangle}, \ref{No4cycle}, and \ref{distance>=3}, $y_2$ and $y_3$ are distinct, not on this path, nonadjacent, and are 3-vertices.  Let $G' = G - \{x_2,x_3\} + y_2y_3$.  Observe that $G'$ is a subcubic planar multigraph with multiplicity at most two.  Thus, by the minimality of $G$, $G'$ has a a good coloring.  We impose this coloring onto $G$ by coloring $x_2y_2$ and $x_3y_3$ with the color used on $y_2y_3$.  Let $\phi$ denote this good partial coloring of $G$.

Let $\alpha := \phi(x_0x_1)$ and $\beta := \phi(x_4x_5)$.  Without loss of generality, suppose $\sU_\phi(y_2) = \{1,2,3\}$ with $\phi(x_2y_2) = \phi(x_3y_3) = 1$.  We may also suppose that $\beta \neq 2$.  By the existence of $y_1y_2$ in $G'$, $\sU_\phi(y_3) \neq \{1,2,3\}$.  Therefore, we color $x_2x_3$ and $x_3x_4$ with 3 and 2, respectively, and then color $x_1x_2$ with a color from $\{4,5\}\setminus\{\alpha\}$.  This yields a good coloring of $G$.
\end{proof}

\begin{lemma}\label{distance>=5}
If the boundary of a face in $G$ contains a pair of 2-vertices, then the distance on the boundary between them is at least five.
\end{lemma}
\begin{proof}
By Lemma \ref{distance>=4}, any face contradicting the statement has length at least eight and contains a path $x_0x_1x_2x_3x_4x_5x_6$ such that $x_1$ and $x_5$ are 2-vertices.  By Lemma \ref{distance>=4}, all other $x_i$ are 3-vertices, and so, for $j \in \{2,3,4\}$ we let $y_j$ be the third neighbor of $x_j$.   By Lemmas \ref{NoTriangle}, \ref{separating}, \ref{No4cycle}, and \ref{distance>=4},  $y_3,y_4, y_5$ are not on this path, distinct, pairwise nonadjacent, and are 3-vertices.

Let $G'$ be obtained from $G$ by removing $x_1, x_2, x_3, x_4, x_5$, adding a new vertex $x$, and adding the edges $xy_2, xy_3, xy_4, x_0x_6$.  Observe that $G'$ is a subcubic planar multigraph with multiplicity at most two.  Thus, by the minimality of $G$, $G'$ has a good coloring.  We impose this good coloring onto $G$ by coloring $x_iy_i$ with the same color used on $xx_i$, and coloring $x_0x_1$ and $x_5x_6$ with the same color used on $x_0x_6$.  Let $\phi$ denote this good partial coloring of $G$.

Without loss of generality assume $\phi(x_iy_i) = i$ for $i \in \{2,3,4\}$, and let $\alpha := \phi(x_0x_1) = \phi(x_5x_6)$.  By the construction of $G'$, $\sU_\phi(y_3) \neq \{2,3,4\}$.   Suppose $\alpha \in \{2,3,4\}$.  We then color $x_2x_3$ and $x_3x_4$ with 4 and 2, respectively.  Now color $x_4x_5$ from $\{1,5\}$ with respect to 2 and $\sU_\phi(y_4)$, and color $x_1x_2$ with a color from $\{1,5\}$ with respect to 4 and $\sU_\phi(y_2)$.  This yields a good coloring of $G$.

So without loss of generality, $\alpha = 1$.  Suppose $\sU_\phi(y_2) \neq \{2,4,5\}$.  By the construction of $G'$, $\sU_\phi(y_4) \neq \{2,3,4\}$.  So we color $x_3x_4$ and $x_4x_5$ with 2 and 3, respectively.  We then color $x_2x_3$ (and $x_1x_2$) from $\{4,5\}$ with respect to 2 and $\sU_\phi(y_3)$.  This yields a good coloring of $G$.

So $\sU_\phi(y_2) = \{2,4,5\}$, and by a symmetric argument $\sU_\phi(y_4) = \{2,4,5\}$.  We now color $x_1x_2, x_2x_3, x_4x_5$ with 4, 1, 3, respectively, and color $x_3x_4$ from $\{2,5\}$ with respect to 1 and $\sU_\phi(y_3)$.  This yields a good coloring of $G$.
\end{proof}

\section{Faces Without 2-Vertices}\label{struct2}

In this section, we show that if a face has a 2-vertex, then that face must have length at least eight.

\begin{lemma}\label{No2on5cycle}
Every vertex of a 5-cycle in $G$ is a 3-vertex.
\end{lemma}
\begin{proof}
By Lemma \ref{separating}, it suffices to consider 5-faces.  Suppose on the contrary that $x_0x_1x_2x_3x_4$ is a 5-face in $G$ and $x_0$ is a $2$-vertex.   Lemma \ref{distance>=4} implies that each $x_i$ other than $x_0$ has a third neighbor $y_i$.   By Lemmas \ref{NoTriangle}, \ref{separating} and \ref{No4cycle}, these $y_i$ are distinct, not on our cycle and pairwise nonadjacent except for possibly $y_1y_4$.  Furthermore, each $y_i$ is a 3-vertex by Lemma \ref{distance>=4}.

Let $G' = G - \{x_0,x_1, x_2, x_3, x_4\} + y_1y_2 + y_3y_4$.  Observe that $G'$ is a subcubic planar multigraph with multiplicity at most two.  Thus, by the minimality of $G$, $G'$ has a good coloring.  We impose this coloring onto $G$ by coloring $x_1y_1$ and $x_2y_2$ with the color used on $y_1y_2$, and coloring $x_3y_3$ and $x_4y_4$ with the color used on $y_3y_4$.  Let $\phi$ denote this good partial coloring of $G$.  

Without loss of generality, suppose $\sU_\phi(y_3) = \{1,2,3\}$ with $\phi(x_3y_3) = \phi(x_4y_4) = 1$.  By the construction of $G'$, $\sU_\phi(y_4) \neq \{1,2,3\}$. 

\begin{case}
$\phi(x_1y_1) = \phi(x_2y_2) \in \{1,2\}$.
\end{case}

Suppose $\sU_\phi(y_2) \neq \{1,4,5\}$.  Color $x_0x_1, x_3x_4, x_4x_0$ with 3, 3, 2, respectively.  We then color $x_1x_2$ (and $x_2x_3$) from $\{4,5\}$ with respect to 3 and $\sU_\phi(y_1)$.  This yields a good coloring of $G$.

So $\sU_\phi(y_2) = \{1,4,5\}$, and by the contruction of $G$, $\sU_\phi(y_1) \neq \{1,4,5\}$.  Now $\sU_\phi(y_4) = \{1,4,5\}$, otherwise color the cycle in order with 4, 5, 3, 4, 5.  We then color the cycle in order with 4, 5, 3, 4, 2.  These are good colorings of $G$ and prove the case.

\begin{case}
$\phi(x_1y_1) = \phi(x_2y_2) = 4$.
\end{case}

Suppose $\sU_\phi(y_1) \neq \{1,2,4\}$.  Color $x_2x_3, x_3x_4, x_4x_0$ with 5, 2, 3, respectively.  We then color $x_1x_2$ (and $x_0x_1$) from $\{1,2\}$ with respect to 5 and $\sU_\phi(y_2)$.  This yields a good coloring of $G$.

So $\sU_\phi(y_1) = \{1,2,4\}$.  Color $x_0x_1, x_1x_2, x_2x_3, x_3x_4$ with 3, 1, 2, 3, respectively.  We then color $x_4x_0$ from $\{4,5\}$ with respect to 3 and $\sU_\phi(y_4)$.  This yields a good coloring of $G$.
\end{proof}

\begin{lemma}\label{No2on6cycle}
Every vertex of a 6-cycle in $G$ is a 3-vertex. 
\end{lemma}
\begin{proof}
Suppose that $G$ has a $6$-cycle $C$ given by $x_0x_1x_2x_3x_4x_5$ on which $x_0$ is a 2-vertex.  By Lemma \ref{distance>=5}, $x_0$ is the only 2-vertex of $C$. 

\setcounter{case}{0}
\begin{case}
$C$ is a separating 6-cycle.
\end{case}

By Lemmas \ref{NoTriangle}, \ref{separating} and \ref{No4cycle}, $C$ has no chords.  As $G$ is subcubic, each vertex of $C$ is incident to at most one edge not on $C$.   Since $\lfloor \frac{5}{2}\rfloor =  2$, by symmetry we may assume that there are at most two edges inside $C$ that are incident to vertices on $C$ (recall that $G$ is assumed to be embedded in the plane).  If there is exactly one such edge, then $G$ has a cut-edge, contradicting Lemma \ref{cut-vertex}.  So, we  have two such edges, and by Lemma \ref{2-edge-cut} these edges  share a common endpoint, say $u$, inside of $C$.  Now, $u$ is a 2-vertex, else it is a cut-vertex with a cut-edge.  However, $u$ together with the vertices of $C$ contains either a triangle, a 4-cycle, or a 5-cycle containing a 2-vertex, contradicting Lemmas \ref{NoTriangle}, \ref{No4cycle}, \ref{separating}, or \ref{No2on5cycle}, respectively.  

\begin{case}
$C$ is not a separating 6-cycle.
\end{case}

Recall that $G$ is assumed to be embedded into the plane.  Thus $C$ must be the boundary of a 6-face.  As mentioned above, each $x_i$, other than $x_0$, is a 3-vertex and so has a third neighbor $y_i$.  We claim that these $y_i$'s are distinct, pairwise disjoint and not on $C$.  Indeed, if any $y_i$ was on $C$, we would create either a triangle or 4-cycle, contradicting Lemmas \ref{NoTriangle} and \ref{No4cycle}.  For $i \in \{1,2,3,4\}$, if $y_i = y_{i+1}$, we have a triangle contradicting Lemma \ref{NoTriangle}.  For $i \in \{1,2,3,5\}$ taken modulo 5, if $y_i = y_{i+2}$, we have a 4-cycle contradicting Lemma \ref{No4cycle}.  For $i \in \{1,2\}$, if $y_i = y_{i+3}$, then $y_ix_ix_{i+1}x_{i+2}x_{i+3}y_{i+3}$ is a separating 5-cycle contradicting Lemma \ref{separating}.  Thus, the $y_i$'s are distinct.  For $i \in \{1,2,3,4\}$, if $y_iy_{i+1} \in E(G)$, we have a 4-cycle contradicting Lemma \ref{No4cycle}.  For $i \in \{1,2,3\}$ if $y_iy_{i+2} \in E(G)$, we have a separating 5-cycle contradicting Lemma \ref{separating}.  If $y_5y_1 \in E(G)$, then $y_1x_1x_0x_5y_5y_1$ is a 5-cycle containing a 2-vertex contradicting Lemma \ref{No2on5cycle}.  For $i \in \{1,2\}$ if $y_iy_{i+3} \in E(G)$, then $y_ix_ix_{i+1}x_{i+2}x_{i+3}y_{i+3}y_i$ is a separating 6-cycle contradicting Case 1.  Thus, the $y_i$'s are pairwise disjoint.  Furthermore, by Lemma \ref{distance>=4}, the only possible 2-vertex amongst the $y_i$'s is $y_3$.

Now, let $G'$ be the plane graph formed from $G$ by removing $x_0, x_1, x_2, x_3, x_4, x_5$, and adding the edges $y_1y_2$ and $y_4y_5$.   Observe that $G'$ is a subcubic planar multigraph with multiplicity at most two.  Thus, by the minimality of $G$, it has a good coloring.  We can impose this coloring onto $G$ by coloring $x_1y_1$ and $x_2y_2$ with the color used on $y_1y_2$, and coloring $x_4y_4, x_5y_5$ with the color used on $y_4y_5$.  Call this good partial coloring of $G$, $\phi$.  

Without loss of generality, assume $\phi(x_4y_4) = \phi(x_5y_5) = 1$ and $\sU_\phi(y_4) = \{1,2,3\}$.  By the existence of $y_4y_5$ in $G'$, $\sU_\phi(y_5) \neq \{1,2,3\}$.  Let $\alpha := \phi(x_3y_3)$.  

\begin{subcase}
$\phi(x_1y_1) = \phi(x_2y_2) = 1$.
\end{subcase}

Without loss of generality, we may assume that $\alpha \neq \{2,5\}$.  Suppose $\sU_\phi(y_3) \neq \{\alpha, 2, 5\}$.  Color $x_2x_3, x_3x_4, x_4x_5, x_5x_0$ with 2, 5, 3, 2, respectively.  We then color $x_1x_2$ from $\{3,4\}$ with respect to 2 and $\sU_\phi(y_2)$.  Let $\beta \in \{3,4\}$ denote the color used on $x_1x_2$.  We then color $x_0x_1$ from $\{3,4,5\}\setminus\{\beta\}$ with respect to $\beta$ and $\sU_\phi(y_1)$.  This yields a good coloring of $G$.

So $\sU_\phi(y_3) = \{\alpha, 2, 5\}$.  If $\alpha = 1$, color $x_2x_3$ and $x_3x_4$ with 4 and 5, respectively.  We then color $x_1x_2$ from $\{2,5\}$ with respect to 4 and $\sU_\phi(y_2)$.  Let $\gamma$ denote the color used on $x_1x_2$, and color $x_0x_1$ from $\{2,3,5\}\setminus\{\gamma\}$ with respect to $\gamma$ and $\sU_\phi(y_1)$.  We then color $x_4x_5$ and $x_5x_0$ properly from $\{2,3\}$.  This yields a good coloring of $G$.

So $\alpha \in \{3,4\}$, and there exists $\oalpha$ such that $\{\alpha, \oalpha\} = \{3,4\}$.  Color $x_2x_3$ and $x_3x_4$ with $\oalpha$ and 5, respectively.  We then color $x_1x_2$ from $\{2,5\}$ with respect to $\oalpha$ and $\sU_\phi(y_2)$.  Let $\gamma \in \{2,5\}$ denote the color used on $x_1x_2$, and color $x_0x_1$ from $\{\alpha,2,5\}\setminus\{\gamma\}$ with respect to $\gamma$ and $\sU_\phi(y_1)$.  We then color $x_4x_5$ and $x_5x_0$ properly from $\{2,3\}$.  This yields a good coloring of $G$.

\begin{subcase}
$\phi(x_1y_1) = \phi(x_2y_2) = 2$.
\end{subcase}

\begin{subsubcase}
$\alpha \in \{1,2\}$.
\end{subsubcase}

Suppose $\sU_\phi(y_3) \neq \{\alpha, 4, 5\}$.  Color $x_2x_3, x_3x_4, x_4x_5, x_5x_0$ with 5, 4, 3, 2, respectively.  Color $x_1x_2$ from $\{1,3\}$ with respect to 5 and $\sU_\phi(y_2)$.  Let $\beta$ denote the color used on $x_1x_2$.  We then color $x_0x_1$ from $\{1,3,4\}\setminus\{\beta\}$ with respect to $\beta$ and $\sU_\phi(y_1)$.  This yields a good coloring of $G$.

So $\sU_\phi(y_3) = \{\alpha, 4, 5\}$.  Suppose $\sU_\phi(y_1) \neq \{2,4,5\}$.  Color $x_2x_3, x_4x_5, x_5x_0$ with 3, 2, 3, respectively.  Color $x_1x_2$ (and $x_0x_1$ and $x_3x_4$) from $\{4,5\}$ with respect to 3 and $\sU_\phi(y_2)$.  This yields a good coloring of $G$.  

So $\sU_\phi(y_1) = \{2,4,5\}$.  Color $x_0x_1, x_2x_3, x_4x_5, x_5x_0$ with 1, 3, 2, 3, respectively.  We then color $x_1x_2$ (and $x_3x_4$) from $\{4,5\}$ with respect to 3 and $\sU_\phi(y_2)$.  This yields a good coloring of $G$.

\begin{subsubcase}
$\alpha = 3$.
\end{subsubcase}

Suppose $\sU_\phi(y_3) \neq \{3,4,5\}$.  Color $x_1x_2, x_4x_5, x_5x_0$ with 1, 3, 2, respectively.  Color $x_2x_3$ (and $x_3x_4$) from $\{4,5\}$ with respect to 1 and $\sU_\phi(y_2)$.  Let $\beta$ denote the color used on $x_2x_3$, and color $x_0x_1$ from $\{3,4,5\}\setminus\{\beta\}$ with respect to 1 and $\sU_\phi(y_1)$.  This yields a good coloring of $G$.

So $\sU_\phi(y_3) = \{3,4,5\}$.  Suppose $\sU_\phi(y_1) \neq \{2,4,5\}$.  Color $x_2x_3, x_3x_4, x_4x_5, x_5x_0$ with 1, 4, 2, 3, respectively.  Color $x_1x_2$ (and $x_0x_1$) from $\{4,5\}$ with respect to 1 and $\sU_\phi(y_2)$.  This yields a good coloring of $G$.

So $\sU_\phi(y_1) = \{2,4,5\}$.  Then $\sU_\phi(y_2) = \{1,2,3\}$, otherwise color the cycle in order with 4, 3, 1, 4, 2, 3.  This is a good coloring of $G$.  We then color $x_0x_1, x_1x_2, x_2x_3, x_4x_5$ with 3, 4, 1, 2, respectively.  We then color $x_5x_0$ (and $x_3x_4$) from $\{4,5\}$ with respect to 2 and $\sU_\phi(y_5)$.  This yields a good coloring of $G$.

\begin{subsubcase}
$\alpha \in \{4,5\}$.
\end{subsubcase}

Let $\oalpha$ be such that $\{\alpha, \oalpha\} = \{4,5\}$. Suppose $\sU_\phi(y_3) \neq \{3,4,5\}$.  Color $x_2x_3, x_3x_4, x_4x_5, x_5x_0$ with 3, $\oalpha$, 3, 2, respectively.  Now color $x_1x_2$ from $\{4,5\}$ with respect to 3 and $\sU_\phi(y_2)$.  Let $\beta$ denote the color used on $x_1x_2$.  We then color $x_0x_1$ from $\{1,4,5\}\setminus\{\beta\}$ with respect to $\beta$ and $\sU_\phi(y_1)$.  This yields a good coloring of $G$.

So $\sU_\phi(y_3) = \{3,4,5\}$.  Color $x_2x_3, x_3x_4, x_4x_5, x_5x_0$ with 1, $\oalpha$, 3, 2, respectively.  Now color $x_1x_2$ from $\{4,5\}$ with respect to 1 and $\sU_\phi(y_2)$.  Let $\gamma$ denote the color used on $x_1x_2$.  We then color $x_0x_1$ from $\{3,4,5\}\setminus\{\gamma\}$ with respect to $\gamma$ and $\sU_\phi(y_1)$.  This yields a good coloring of $G$.

\begin{subcase}
$\phi(x_1y_1) = \phi(x_2y_2) = 4$.
\end{subcase}

\begin{subsubcase}
$\alpha \in \{1,4,5\}$.
\end{subsubcase}

If $\alpha = 5$, suppose $\sU_\phi(y_3) \neq \{2,4,5\}$.  Color $x_2x_3, x_3x_4, x_4x_5, x_5x_0$ with 2, 4, 3, 2, respectively.  Now color $x_1x_2$ from $\{1,3\}$ with respect to 2 and $\sU_\phi(y_2)$.  Let $\beta$ denote the color used on $x_1x_2$, and color $x_0x_1$ from $\{1,3,5\}\setminus\{\beta\}$ with respect to $\beta$ and $\sU_\phi(y_1)$.  This yields a good coloring of $G$.  

If $\alpha \in \{1,4\}$, suppose $\sU_\phi(y_3) \neq \{\alpha, 2, 5\}$.  Color $x_2x_3, x_3x_4, x_4x_5, x_5x_0$ with 2, 5, 3, 2, respectively.  Now color $x_1x_2$ from $\{1,3\}$ with respecto to 2 and $\sU_\phi(y_2)$.  Let $\gamma$ denote the color used on $x_1x_2$, and color $x_0x_1$ from $\{1,3,5\}\setminus\{\gamma\}$ with respect to $\gamma$ and $\sU_\phi(y_1)$.  This yields a good coloring of $G$.

\begin{subsubcase}
$\alpha \in \{2,3\}$.
\end{subsubcase}

Let $\oalpha$ be such that $\{\alpha, \oalpha\} = \{2,3\}$.  Suppose $\sU_\phi(y_3) \neq \{2,3,5\}$.  Color $x_2x_3, x_3x_4, x_4x_5$ with $\oalpha$, 5, $\alpha$, $\oalpha$, respectively.  Now color $x_1x_2$ from $\{1,5\}$ with respect to $\oalpha$ and $\sU_\phi(y_2)$.  Let $\beta$ denote the color used on $x_1x_2$, and color $x_0x_1$ from $\{1,\alpha,5\}\setminus\{\beta\}$ with respect to $\beta$ and $\sU_\phi(y_1)$.  This yields a good coloring of $G$.

So $\sU_\phi(y_3) = \{2,3,5\}$.  Suppose $\sU_\phi(y_2) \neq \{1,\alpha,4\}$.  Color $x_1x_2, x_2x_3, x_3x_4, x_4x_5, x_5x_0$ with $\alpha$, 1, 5, $\oalpha$, $\alpha$, respectively.  We then color $x_0x_1$ from $\{\oalpha, 5\}$ with respect to $\alpha$ and $\sU_\phi(y_1)$.  This yields a good coloring of $G$.

So $\sU_\phi(y_2) = \{1, \alpha, 4\}$.  We then color the cycle in order with $\alpha$, 1, 5, 4, $\alpha$, $\oalpha$.  This is a good coloring of $G$.

This exhausts all cases and proves the lemma.
\end{proof}

\begin{lemma}\label{No2on7face}
Every vertex of a 7-face in $G$ is a 3-vertex.
\end{lemma}

\begin{proof}
Recall that $G$ is assumed to be embedded into the plane.  Suppose on the contrary that $G$ has a 7-face with boundary $x_0x_1x_2\dots x_6$ with $x_0$ being a 2-vertex.  By Lemma \ref{distance>=5}, each $x_i$ other than $x_0$ has a third neighbor $y_i \notin \{x_{i-1},x_{i+1}\}$ where $i$ is taken modulo 7.  Similarly to  Case 2 of Lemma \ref{No2on6cycle},  Lemmas \ref{NoTriangle}, \ref{No4cycle}, \ref{separating}, \ref{No2on5cycle} and \ref{No2on6cycle}, imply that the $y_i$'s are not on the 7-face, are distinct and the only possible adjacencies other than those on this face or $x_iy_i$, $i \in \{1,2,3,4,5,6\}$, are $y_1y_4, y_2y_5, y_3y_6$.  In particular, $y_2y_6, y_1y_5 \notin E(G)$ by Lemma \ref{No2on6cycle}.

Let $G'$ be obtained from $G$ by removing $x_0, x_1, \dots, x_6$ and adding the edges $y_1y_2, y_3y_4, y_5y_6$.  Observe that $G'$ is a subcubic planar multigraph with multiplicity at most two, and so by the minimality of $G$, it has a good coloring.  We impose this coloring onto $G$ by coloring $x_1y_1$ and $x_2y_2$ with the color used on $y_1y_2$, coloring $x_3y_3$ and $x_4y_4$ with the color used on $y_3y_4$, and coloring $x_5y_5$ and $x_6y_6$ with the color used on $y_5y_6$.  Without loss of generality assume $\phi(x_1y_1) = \phi(x_2y_2) = 1$.  

\setcounter{case}{0}
\begin{case}
$(\phi(x_3y_3), \phi(x_5y_5)) = (1,1)$.
\end{case}

We proceed in each of the following cases attempting to color our cycle in some order.  Color $x_1x_2$ with $\alpha_{12} \notin \sU_\phi(y_1)$.  We then color $x_2x_3$ with $\alpha_{23} \notin (\sU_\phi(y_2) \cup \{\alpha_{12}\})$.  For $i \in \{3,4,5,6\}$ taken modulo 7, we color $x_ix_{i+1}$ from $\{1,2,3,4,5\}\setminus\{1,\alpha_{(i-2)(i-1)}, \alpha_{(i-1)i}\}$ with respect to $\alpha_{(i-1)i}$ and $\sU_\phi(y_i)$, and let $\alpha_{i(i+1)}$ denote the color used on $x_ix_{i+1}$.

Thus, it only remains to color $x_0x_1$.  Since $\alpha_{12} \notin \sU_\phi(y_1)$, it suffices to color $x_0x_1$ with a color not in $\{\alpha_{60}, \alpha_{12}, \alpha_{23},1\}$.  This yields a good coloring of $G$.

\begin{case}
$(\phi(x_3y_3), \phi(x_5y_5)) = (1,2)$.
\end{case}

In this case, we begin by coloring $x_5x_6$ with $\alpha_{56} \notin (\sU_\phi(y_6) \cup \{1\})$.  We then color $x_4x_5$ from $\{1,2,3,4,5\}\setminus\{1,2,\alpha_{56}\}$ with respect to $\alpha_{56}$ and $\sU_\phi(y_5)$.  Let $\alpha_{45}$ denote the color used on $x_4x_5$.  We now color $x_2x_3, x_1x_2, x_0x_1, x_6x_0$ in this order, in a manner similar to the previous case.  This yields a good coloring of $G$.

\begin{case}
$(\phi(x_3y_3), \phi(x_5y_5)) = (2,\beta)$.
\end{case}

Up to relabeling, we may assume that $\beta \in \{1,3\}$.  In this case, we begin by coloring $x_1x_2$ with $\alpha_{12} \notin (\sU_\phi(y_1) \cup \{2\})$.  We then color $x_2x_3$ from $\{1,2,3,4,5\}\setminus\{1,2,\alpha_{12}\}$ with respect to $\alpha_{12}$ and $\sU_\phi(y_2)$. Let $\alpha_{23}$ denote the color used on $x_2x_3$, and color $x_3x_4$ from $\{1,2,3,4,5\}\setminus\{2,\alpha_{23},\beta\}$ with respect to $\alpha_{23}$ and $\sU_\phi(y_3)$.  This results in a good partial coloring of $G$.  Call it $\sigma$.

Suppose that we can continue our good partial coloring of $G$ by coloring $x_4x_5$ with some $\alpha_{45}$.  We now color $x_5x_6$ from $\{1,2,3,4,5\}\setminus\{2,\alpha_{45}, \beta\}$ with respect to $\alpha_{45}$ and $\sU_\sigma(y_5)$.  Let $\alpha_{56}$ denote the color used on $x_5x_6$.  We then color $x_6x_0$ from $\{1,2,3,4,5\}\setminus\{\alpha_{45}, \alpha_{56}, \beta\}$ with respect to $\alpha_{56}$ and $\sU_\sigma(y_6)$.  We end by coloring $x_0x_1$ with a color not in $\{1,\alpha_{12}, \alpha_{23}, \alpha_{60}\}$.  This yields a good coloring of $G$.

So we assume that we cannot extend $\sigma$ when  attempting to color $x_4x_5$.  As a result, $\alpha_{34} \in \sU_\sigma(y_4)$, otherwise we could color $x_4x_5$ with a color not in $\{2, \beta, \alpha_{23}, \alpha_{34}\}$.  Similarly, if $|\sU_\sigma(y_4) \cup \{\alpha_{23},\beta\}| \le 4$, we can color $x_4x_5$ with a color not in $\sU_\sigma(y_4) \cup \{\alpha_{23}, \beta\}$.   Thus, $y_4$ is a 3-vertex with $\sU_\sigma(y_4) = \{2,\alpha_{34}, \gamma\}$, where $\gamma$ is such that $\{1,2,3,4,5\} = \{2, \beta, \gamma, \alpha_{23}, \alpha_{34}\}$.

We now uncolor $x_3x_4$ and relabel our colors so that $\{\alpha_{34}, \gamma\} = \{\gamma_1, \gamma_2\}$.  So $\{2,\beta, \gamma_1, \gamma_2, \alpha_{23}\} = \{1,2,3,4,5\}$.  Suppose $\sU_\phi(y_3) \neq \{2, \beta, \alpha_{23}\}$.  We then color $x_3x_4$ with $\beta$, $x_5x_6$ with a color $\alpha_{56} \notin \sU_\phi(y_5) \cup \{2\}$, and $x_4x_5$ with a color $\alpha_{45} \in \{\gamma_1, \gamma_2\}\setminus\{\alpha_{56}\}$.  This yields a good partial coloring of $G$ that we can extend to $x_6x_0$ and $x_0x_1$ as in the previous cases.

So $\sU_\phi(y_3) = \{2,\beta,\alpha_{23}\}$, and in particular, $\beta \in \sU_\phi(y_3)$.  Now, if we attempt to color our cycle starting from $x_5x_6$ instead of $x_1x_2$, a symmetric argument implies that $1 \in \sU_\phi(y_4)$.  That is, $1 \in \{\gamma_1, \gamma_2\}$.  Since $\{2,\beta, \alpha_{23}, \gamma_1, \gamma_2\} = \{1,2,3,4,5\}$ and $\beta \in \{1,3\}$, this implies that $\beta = 3$.  Without loss of generality, we may assume $\gamma_1 = 1, \gamma_2 = 4, \alpha_{23} = 5$.  This implies that $\alpha_{12} \in \{3,4\}$.  

Suppose $\sU_\phi(y_2) \neq \{1,3,4\}$.  We still assume that $x_1x_2$ is colored with $\alpha_{12} \notin (\sU_\phi(y_1) \cup \{2\})$.  However, we recolor $x_2x_3$ with the color from $\{3,4\}\setminus\{\alpha_{12}\}$.  We then color $x_3x_4$ and $x_4x_5$ with 1 and 5, respectively.  This is a good partial coloring of $G$ from which we can color $x_5x_6, x_6x_0, x_0x_1$ as in the previous cases.

So $\sU_\phi(y_2) = \{1,3,4\}$, and by symmetry $\sU_\phi(y_5) = \{1,3,5\}$.  Color $x_1x_2, x_2x_3, x_3x_4, x_4x_5, x_5x_6$ with 5, 4, 1, 5, 2, respectively.  We then color $x_0x_1$ from $\{2,3\}$ with respect to 5 and $\sU_\phi(y_1)$, and color $x_6x_0$ from $\{1,4\}$ with respect to 2 and $\sU_\phi(y_6)$.  This yields a good coloring of $G$.
\end{proof}

\section{Adjacent Faces}\label{struct3}

By the lemmas in Section \ref{struct1}, every face in $G$ is a $5^+$-face.  In this section we show that if a face has length five, then it can only be adjacent to $7^+$-faces.  The proofs of these lemmas are simply detailed case analysis.  Thus, for the sake of readibility, we omit these details, which can be found in the Appendix.

\begin{figure}[h]
\centering
\begin{tikzpicture}[line cap=round,line join=round,>=triangle 45,x=0.75cm,y=0.75cm]
\clip(-10, -6) rectangle (10,6);
\draw [->] (-1.0,0.0) -- (1.0,0.0);
\draw (-6.0,5.62)-- (-6.0,3.62);
\draw (-6.0,3.62)-- (-7.902113032590307,2.238033988749895);
\draw (-9.804226065180615,2.85606797749979)-- (-7.902113032590307,2.238033988749895);
\draw (-7.902113032590307,2.238033988749895)-- (-7.175570504584947,0.001966011250105315);
\draw (-7.175570504584947,0.001966011250105315)-- (-4.824429495415054,0.001966011250105093);
\draw (-4.824429495415054,0.001966011250105093)-- (-4.097886967409693,2.2380339887498946);
\draw (-4.097886967409693,2.2380339887498946)-- (-5.999999999999999,3.62);
\draw (-2.195773934819386,2.85606797749979)-- (-4.097886967409693,2.238033988749895);
\draw (-4.824429495415054,0.001966011250105093)-- (-4.097886967409693,2.2380339887498946);
\draw (-6.000000000000001,-5.6160679774997915)-- (-6.000000000000001,-3.6160679774997897);
\draw (-6.000000000000001,-3.6160679774997897)-- (-7.902113032590307,-2.2341019662496855);
\draw (-9.804226065180615,-2.8521359549995795)-- (-7.902113032590307,-2.2341019662496855);
\draw (-7.902113032590307,-2.2341019662496855)-- (-7.175570504584947,0.0019660112501051485);
\draw (-7.175570504584947,0.0019660112501051485)-- (-4.824429495415054,0.001966011250104982);
\draw (-4.824429495415054,0.001966011250104982)-- (-4.097886967409693,-2.234101966249685);
\draw (-4.097886967409693,-2.234101966249685)-- (-6.0,-3.6160679774997897);
\draw (-2.1957739348193863,-2.8521359549995813)-- (-4.097886967409694,-2.2341019662496855);
\draw (-4.824429495415054,0.001966011250104982)-- (-4.097886967409693,-2.234101966249685);
\draw (2.195773934819386,2.85606797749979)-- (6,2);
\draw (6,-2)-- (2.1957739348193854,-2.8521359549995795);
\draw (9.804226065180615,2.85606797749979)-- (6,2);
\draw (6,-2)-- (9.804226065180615,-2.8521359549995813);
\draw (6.0,5.62)-- (6,2);
\draw (6.0,-2)-- (5.999999999999999,-5.6160679774997915);
\draw (-5.9,5.8) node[anchor=north west] {$y_2$};
\draw (6.1,5.8) node[anchor=north west] {$y_2$};
\draw (-3,3.6) node[anchor=north west] {$y_3$};
\draw (-3,-2) node[anchor=north west] {$y_5$};
\draw (-5.9,-5.1) node[anchor=north west] {$y_6$};
\draw (-9.7,-2) node[anchor=north west] {$y_7$};
\draw (-9.7,3.6) node[anchor=north west] {$y_1$};
\draw (9,3.6) node[anchor=north west] {$y_3$};
\draw (2.3,3.6) node[anchor=north west] {$y_1$};
\draw (9,-2) node[anchor=north west] {$y_5$};
\draw (6.1, -5.1) node[anchor=north west] {$y_6$};
\draw (2.3,-2) node[anchor=north west] {$y_7$};
\draw (-6.4,3.5) node[anchor=north west] {$x_2$};
\draw (-5,2.5) node[anchor=north west] {$x_3$};
\draw (-5.6,.6) node[anchor=north west] {$x_4$};
\draw (-7.3, .6) node[anchor=north west] {$x_0$};
\draw (-7.8,2.5) node[anchor=north west] {$x_1$};
\draw (-7.8, -1.8) node[anchor=north west] {$x_7$};
\draw (-6.4,-2.8) node[anchor=north west] {$x_6$};
\draw (-5,-1.8) node[anchor=north west] {$x_5$};
\draw (5.7,1.9) node[anchor=north west] {$u$};
\draw (5.7,-1.3) node[anchor=north west] {$v$};
\begin{scriptsize}
\draw [fill=black] (-6.0,3.62) circle (1.5pt);
\draw [fill=black] (-7.902113032590307,2.238033988749895) circle (1.5pt);
\draw [fill=black] (-7.175570504584947,0.001966011250105315) circle (1.5pt);
\draw [fill=black] (-4.824429495415054,0.001966011250105093) circle (1.5pt);
\draw [fill=black] (-4.097886967409693,2.2380339887498946) circle (1.5pt);
\draw [fill=black] (-6.0,5.62) circle (1.5pt);
\draw [fill=black] (6.0,2) circle (1.5pt);
\draw [fill=black] (6.0,-2) circle (1.5pt);
\draw [fill=black] (-9.804226065180615,2.85606797749979) circle (1.5pt);
\draw [fill=black] (-7.902113032590307,2.238033988749895) circle (1.5pt);
\draw [fill=black] (-7.902113032590307,2.238033988749895) circle (1.5pt);
\draw [fill=black] (-7.175570504584947,0.001966011250105315) circle (1.5pt);
\draw [fill=black] (-7.175570504584947,0.001966011250105315) circle (1.5pt);
\draw [fill=black] (-7.175570504584947,0.001966011250105315) circle (1.5pt);
\draw [fill=black] (-4.824429495415054,0.001966011250105093) circle (1.5pt);
\draw [fill=black] (-7.175570504584947,0.001966011250105315) circle (1.5pt);
\draw [fill=black] (-4.824429495415054,0.001966011250105093) circle (1.5pt);
\draw [fill=black] (-4.824429495415054,0.001966011250105093) circle (1.5pt);
\draw [fill=black] (-4.824429495415054,0.001966011250105093) circle (1.5pt);
\draw [fill=black] (-4.097886967409693,2.2380339887498946) circle (1.5pt);
\draw [fill=black] (-4.824429495415054,0.001966011250105093) circle (1.5pt);
\draw [fill=black] (-4.097886967409693,2.2380339887498946) circle (1.5pt);
\draw [fill=black] (-4.824429495415054,0.001966011250105093) circle (1.5pt);
\draw [fill=black] (-4.824429495415054,0.001966011250105093) circle (1.5pt);
\draw [fill=black] (-4.097886967409693,2.2380339887498946) circle (1.5pt);
\draw [fill=black] (-4.097886967409693,2.2380339887498946) circle (1.5pt);
\draw [fill=black] (-4.097886967409693,2.2380339887498946) circle (1.5pt);
\draw [fill=black] (-5.999999999999999,3.62) circle (1.5pt);
\draw [fill=black] (-4.097886967409693,2.2380339887498946) circle (1.5pt);
\draw [fill=black] (-5.999999999999999,3.62) circle (1.5pt);
\draw [fill=black] (-4.097886967409693,2.2380339887498946) circle (1.5pt);
\draw [fill=black] (-4.097886967409693,2.2380339887498946) circle (1.5pt);
\draw [fill=black] (-4.097886967409693,2.2380339887498946) circle (1.5pt);
\draw [fill=black] (-5.999999999999999,3.62) circle (1.5pt);
\draw [fill=black] (-4.097886967409693,2.2380339887498946) circle (1.5pt);
\draw [fill=black] (-4.097886967409693,2.2380339887498946) circle (1.5pt);
\draw [fill=black] (-5.999999999999999,3.62) circle (1.5pt);
\draw [fill=black] (-2.195773934819386,2.85606797749979) circle (1.5pt);
\draw [fill=black] (-4.097886967409693,2.238033988749895) circle (1.5pt);
\draw [fill=black] (-4.824429495415054,0.001966011250105093) circle (1.5pt);
\draw [fill=black] (-4.097886967409693,2.2380339887498946) circle (1.5pt);
\draw [fill=black] (-4.824429495415054,0.001966011250105093) circle (1.5pt);
\draw [fill=black] (-4.824429495415054,0.001966011250105093) circle (1.5pt);
\draw [fill=black] (-4.824429495415054,0.001966011250105093) circle (1.5pt);
\draw [fill=black] (-4.824429495415054,0.001966011250105093) circle (1.5pt);
\draw [fill=black] (-4.824429495415054,0.001966011250105093) circle (1.5pt);
\draw [fill=black] (-4.824429495415054,0.001966011250105093) circle (1.5pt);
\draw [fill=black] (-4.097886967409693,2.2380339887498946) circle (1.5pt);
\draw [fill=black] (-4.824429495415054,0.001966011250105093) circle (1.5pt);
\draw [fill=black] (-4.824429495415054,0.001966011250105093) circle (1.5pt);
\draw [fill=black] (-4.824429495415054,0.001966011250105093) circle (1.5pt);
\draw [fill=black] (-4.097886967409693,2.2380339887498946) circle (1.5pt);
\draw [fill=black] (-4.824429495415054,0.001966011250105093) circle (1.5pt);
\draw [fill=black] (-4.824429495415054,0.001966011250105093) circle (1.5pt);
\draw [fill=black] (-4.097886967409693,2.2380339887498946) circle (1.5pt);
\draw [fill=black] (-6.000000000000001,-5.6160679774997915) circle (1.5pt);
\draw [fill=black] (-6.000000000000001,-3.6160679774997897) circle (1.5pt);
\draw [fill=black] (-6.000000000000001,-3.6160679774997897) circle (1.5pt);
\draw [fill=black] (-7.902113032590307,-2.2341019662496855) circle (1.5pt);
\draw [fill=black] (-9.804226065180615,-2.8521359549995795) circle (1.5pt);
\draw [fill=black] (-7.902113032590307,-2.2341019662496855) circle (1.5pt);
\draw [fill=black] (-7.902113032590307,-2.2341019662496855) circle (1.5pt);
\draw [fill=black] (-7.175570504584947,0.0019660112501051485) circle (1.5pt);
\draw [fill=black] (-7.175570504584947,0.0019660112501051485) circle (1.5pt);
\draw [fill=black] (-4.824429495415054,0.001966011250104982) circle (1.5pt);
\draw [fill=black] (-4.824429495415054,0.001966011250104982) circle (1.5pt);
\draw [fill=black] (-4.097886967409693,-2.234101966249685) circle (1.5pt);
\draw [fill=black] (-4.097886967409693,-2.234101966249685) circle (1.5pt);
\draw [fill=black] (-6.0,-3.6160679774997897) circle (1.5pt);
\draw [fill=black] (-2.1957739348193863,-2.8521359549995813) circle (1.5pt);
\draw [fill=black] (-4.097886967409694,-2.2341019662496855) circle (1.5pt);
\draw [fill=black] (-4.824429495415054,0.001966011250104982) circle (1.5pt);
\draw [fill=black] (-4.097886967409693,-2.234101966249685) circle (1.5pt);
\draw [fill=black] (-7.175570504584947,0.0019660112501051485) circle (1.5pt);
\draw [fill=black] (-4.824429495415054,0.001966011250104982) circle (1.5pt);
\draw [fill=black] (-4.097886967409693,-2.234101966249685) circle (1.5pt);
\draw [fill=black] (-7.175570504584947,0.0019660112501051485) circle (1.5pt);
\draw [fill=black] (-7.175570504584947,0.0019660112501051485) circle (1.5pt);
\draw [fill=black] (-4.824429495415054,0.001966011250104982) circle (1.5pt);
\draw [fill=black] (-4.824429495415054,0.001966011250104982) circle (1.5pt);
\draw [fill=black] (-4.824429495415054,0.001966011250104982) circle (1.5pt);
\draw [fill=black] (-4.097886967409693,-2.234101966249685) circle (1.5pt);
\draw [fill=black] (-4.824429495415054,0.001966011250104982) circle (1.5pt);
\draw [fill=black] (-4.824429495415054,0.001966011250104982) circle (1.5pt);
\draw [fill=black] (-4.097886967409693,-2.234101966249685) circle (1.5pt);
\draw [fill=black] (-4.097886967409693,-2.234101966249685) circle (1.5pt);
\draw [fill=black] (-4.097886967409693,-2.234101966249685) circle (1.5pt);
\draw [fill=black] (-6.0,-3.6160679774997897) circle (1.5pt);
\draw [fill=black] (-4.097886967409693,-2.234101966249685) circle (1.5pt);
\draw [fill=black] (-4.097886967409693,-2.234101966249685) circle (1.5pt);
\draw [fill=black] (-4.097886967409693,-2.234101966249685) circle (1.5pt);
\draw [fill=black] (-6.0,-3.6160679774997897) circle (1.5pt);
\draw [fill=black] (-4.097886967409693,-2.234101966249685) circle (1.5pt);
\draw [fill=black] (-4.097886967409693,-2.234101966249685) circle (1.5pt);
\draw [fill=black] (-6.0,-3.6160679774997897) circle (1.5pt);
\draw [fill=black] (-4.824429495415054,0.001966011250104982) circle (1.5pt);
\draw [fill=black] (-4.824429495415054,0.001966011250104982) circle (1.5pt);
\draw [fill=black] (-4.824429495415054,0.001966011250104982) circle (1.5pt);
\draw [fill=black] (-4.824429495415054,0.001966011250104982) circle (1.5pt);
\draw [fill=black] (-4.824429495415054,0.001966011250104982) circle (1.5pt);
\draw [fill=black] (-4.824429495415054,0.001966011250104982) circle (1.5pt);
\draw [fill=black] (-4.097886967409693,-2.234101966249685) circle (1.5pt);
\draw [fill=black] (-4.824429495415054,0.001966011250104982) circle (1.5pt);
\draw [fill=black] (-4.824429495415054,0.001966011250104982) circle (1.5pt);
\draw [fill=black] (-4.824429495415054,0.001966011250104982) circle (1.5pt);
\draw [fill=black] (-4.097886967409693,-2.234101966249685) circle (1.5pt);
\draw [fill=black] (-4.824429495415054,0.001966011250104982) circle (1.5pt);
\draw [fill=black] (-4.824429495415054,0.001966011250104982) circle (1.5pt);
\draw [fill=black] (-4.097886967409693,-2.234101966249685) circle (1.5pt);
\draw [fill=black] (2.1957739348193854,-2.8521359549995795) circle (1.5pt);
\draw [fill=black] (5.999999999999999,-5.6160679774997915) circle (1.5pt);
\draw [fill=black] (9.804226065180615,-2.8521359549995813) circle (1.5pt);
\draw [fill=black] (9.804226065180615,2.85606797749979) circle (1.5pt);
\draw [fill=black] (6.0,5.62) circle (1.5pt);
\draw [fill=black] (2.195773934819386,2.85606797749979) circle (1.5pt);
\draw [fill=black] (2.195773934819386,2.85606797749979) circle (1.5pt);
\draw [fill=black] (2.1957739348193854,-2.8521359549995795) circle (1.5pt);
\draw [fill=black] (6.0,5.62) circle (1.5pt);
\draw [fill=black] (5.999999999999999,-5.6160679774997915) circle (1.5pt);
\draw [fill=black] (9.804226065180615,2.85606797749979) circle (1.5pt);
\draw [fill=black] (9.804226065180615,-2.8521359549995813) circle (1.5pt);
\end{scriptsize}
\end{tikzpicture}
\caption{Forming $G'$ from $G$}
\label{fig:No5-5face}
\end{figure}
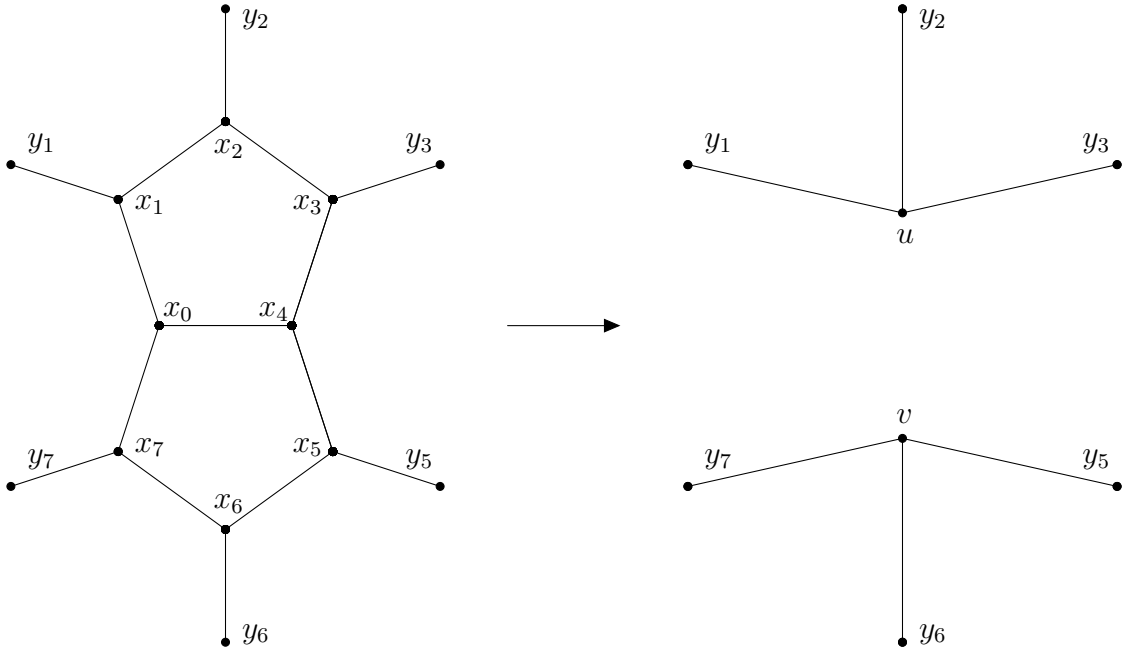

\begin{lemma}\label{No5-5face}
No two 5-faces in $G$ share an edge.
\end{lemma}

\begin{proof}
Suppose the contrary.  By Lemma \ref{No2on5cycle}, the boundaries of the two faces form an 8-cycle, $x_0x_1\dots x_7$ with $x_4x_0 \in E(G)$.  By Lemmas \ref{NoTriangle}, \ref{separating}, \ref{No4cycle}, and \ref{No2on5cycle}, each $x_i$ other than $x_4,x_0$ has a third neighbor $y_i$ not on the 8-cycle that are distinct from each other, except possibly $y_2 = y_6$.  Additionally, the only possible adjacencies between the $y_i$'s are $y_iy_j$ for $i \in \{1,2,3\}$ and $j \in \{5,6,7\}$.  

Let $G'$ denote the graph obtained from $G$ by removing $x_0,\dots, x_7$, adding two new vertices $u,v$ and the edges $uy_1,uy_2, uy_3, vy_5, vy_6, vy_7$ (see Figure \ref{fig:No5-5face}).  Observe that $G'$ is a subcubic planar multigraph with multiplicity at most two, and so by the minimality of $G$, $G'$ has a good coloring.  We impose this coloring onto $G$ by coloring $x_iy_i$ with the same color as $uy_i$ for  $i \in \{1,2,3\}$ and $x_jy_j$ with the same color as $vy_j$, for $j \in \{5,6,7\}$.  Let $\psi$ denote this good partial coloring of $G$.

 By the construction of $G'$, $|\{\psi(x_iy_i): i \in \{1,2,3\}\}| = |\{\psi(x_iy_i): i \in \{5,6,7\}\}| = 3$.  So we can further extend our good partial coloring  by coloring $x_5x_6$ and $x_6x_7$ with $\psi(x_7y_7)$ and $\psi(x_5y_5)$, respectively.  By Lemma \ref{extend}, we can color $x_7x_0$ and $x_4x_5$ as well.  Call this extended, good partial coloring of $G$, $\phi$.  Thus, in remains to color the edges of the 5-cycle $x_0x_1x_2x_3x_4$.   

Without loss of generality, we may assume that $\phi(x_1y_1) = 1$ and $\phi(x_3y_3) = 2$, and let $\alpha: = \phi(x_7x_0)$ and $\beta: = \phi(x_4x_5)$.    By the construction of $G'$, $\{1,2,\phi(x_2y_2)\} \notin \{\sU_\phi(y_i): i \in \{1,2,3\}\}$.  Note that under $\phi$, $\sU_\phi(x_5)\setminus\{\beta\} = \sU_\phi(x_7)\setminus\{\alpha\} =  \{\phi(x_5y_5), \phi(x_7y_7)\}$.   Up to relabeling colors, we may assume that $\{\phi(x_5y_5), \phi(x_7y_7)\} \in \{\{1,2\}, \{1,3\}, \{4,5\}\}$.  

In almost every situation, we can extend $\phi$ to a good coloring of $G$ by case analysis, the details of which can be found in the Appendix.  Thus, we will only consider the situations in which we cannot extend $\phi$ to a good coloring of $G$.

 When $\{\phi(x_5y_5), \phi(x_7y_7)\} = \{1,2\}$, we can always extend $\phi$ to a good coloring of $G$.  

When $\{\phi(x_5y_5), \phi(x_7y_7)\} = \{1,3\}$, we can always extend $\phi$ unless $\phi(x_2y_2) = 3$, $\{\alpha, \beta\} = \{4,5\}$, $\sU_\phi(y_1) = \{1,4,5\}$, $\sU_\phi(y_2) = \{3,4,5\}$, and $\sU_\phi(y_3) = \{1,2,\alpha\}$.  

In this situation, we will reconsider the good partial coloring of $G$ $\psi$.  By the construction of $G'$,  $\phi(x_6y_6) \in  \{2,4,5\}$.  Without loss of generality, assume $\alpha = \phi(x_7x_0) = 4$ and $\beta = \phi(x_4x_5) = 5$.  Thus, $\phi(x_6y_6) = 2$.  If $\sU_\phi(y_5) \neq \{1,3,4\}$, we could recolor $x_4x_5$ with 4.  However, when $\alpha = \beta$, we can extend $\phi$ to a good coloring of $G$.  So $\sU_\phi(y_3) = \{1,3,4\}$, and similarly, $\sU_\phi(y_7) = \{1,3,5\}$.  

We now proceed by reconsidering the good partial coloring of $G$ $\psi$.  Recall that under $\psi$ the edges of the cycle $x_0x_1\dots x_7$ along with the edge $x_4x_0$ are the remaining uncolored edges.  Thus, when we `color the cycle in order' we color the edges $x_0x_1, x_1x_2, \dots, x_6x_7, x_7x_0$ in this order.  

Suppose $\sU_\psi(y_6) \neq \{2,3,5\}$.  If $(\psi(x_5y_5), \psi(x_7y_7)) = (1,3)$, color $x_4x_0$ with 1 and color the cycle in order with 3, 2, 5, 4, 5, 3, 5, 4.  If $(\psi(x_5y_5), \psi(x_7y_7)) = (3,1)$, color $x_4x_0$ with 2 and color the cycle in order with 3, 2, 5, 4, 1, 5, 3, 4.  In either case, this is a good coloring of $G$.

So $\sU_\psi(y_6) = \{2,3,5\}$.  If $(\psi(x_5y_5), \psi(x_7y_7)) = (1,3)$, color $x_4x_0$ with 1 and color the cycle in order with 2, 5, 1, 3, 5, 4, 1, 4.  If $(\psi(x_5y_5), \psi(x_7y_7)) = (3,1)$, color $x_4x_0$ with 4 and color the cycle in order with 2, 5, 1, 3, 5, 1, 4, 3.  In either case, this is a good coloring of $G$.  This proves the case when $\{\phi(x_5y_5), \phi(x_7y_7)\} = \{1,3\}$.

When $\{\phi(x_5y_5), \phi(x_7y_7)\} = \{4,5\}$, we can always extend $\phi$ to a good coloring of $G$, unless up to relableing colors and symmetry, one of two situations occurs.  The first is when $\phi(x_2y_2) = 3$, $(\alpha,\beta)  = (1,3)$, $\sU_\phi(y_1) = \{1,4,5\}$, and $\sU_\phi(y_2) = \{3,4,5\}$. The second is when $\phi(x_2y_2) = 5$, $(\alpha,\beta) = (1,3)$, $\sU_\phi(y_1) = \{1,3,4\}$, and $\sU_\phi(y_2) = \{3,4,5\}$.  In both cases, we reconsider $\psi$ as above.

In a manner similar to the above, we deduce that $\sU_\phi(y_5) = \{1,4,5\}$, $\sU_\phi(y_7) = \{3,4,5\}$, and $\phi(x_6y_6) = 2$.  We now recolor the edges of the cycle $x_0x_1 \dots x_7$ and the edge $x_4x_0$.

If $(\psi(x_5y_5), \psi(x_7y_7)) = (5,4)$, suppose $\sU_\psi(y_6) \neq \{2,3,5\}$.  Then color $x_0x_1, x_1x_2$, $x_3x_4, x_4x_5$, $x_5x_6$, $x_6x_7$, $x_7x_0, x_4x_0$ with 5, 2, 1, 4, 3, 5, 1, 3, respectively, and color $x_2x_3$ from $\{4,5\}$ with respect to 1 and $\sU_\phi(y_3)$.  This yields a good coloring of $G$.  

So $\sU_\phi(y_6) = \{2,3,5\}$.  We then color $x_0x_1, x_1x_2, x_3x_4, x_4x_5, x_5x_6, x_6x_7, x_7x_0, x_4x_0$ with 4, 2, 1, 3, 4, 1, 5, 2, respectively, and color $x_2x_3$ from $\{4,5\}$ with respect to 1 and $\sU_\phi(y_3)$.  This yields a good coloring of $G$.

A similar argument holds for $(\psi(x_5y_5), \psi(x_7y_7)) = (4,5)$ when considering whether or not $\sU_\phi(y_6)$ is $\{2,3,4\}$ by switching the roles of 4 and 5.  This proves the first subcase.

In the second subcase, we again  reconsider the good partial coloring of $G$ $\psi$.  As above, we deduce that $\sU_\phi(y_5) = \{1,4,5\}, \sU_\phi(y_7) = \{3,4,5\}$, and $\phi(x_6y_6) = 2$.  We now recolor the edges of the cycle $x_0x_1\dots x_7$ and the edge $x_4x_0$.

 If $(\psi(x_5y_5), \psi(x_7y_7)) = (5,4)$, suppose $\sU_\psi(y_6) \neq \{1,2,4\}$.  Then color $x_0x_1, x_1x_2$, $x_3x_4$, $x_4x_5$, $x_5x_6, x_6x_7, x_7x_0, x_4x_0$ with 3, 2, 4, 3, 4, 1, 5, 1, respectively, and color $x_2x_3$ from $\{1,3\}$ with respect to 4 and $\sU_\phi(y_3)$.  This yields a good coloring of $G$.

So $\sU_\phi(y_6) = \{1,2,4\}$.  We then color $x_0x_1, x_1x_2, x_3x_4, x_4x_5, x_5x_6, x_6x_7, x_7x_0, x_4x_0$ with 5, 2, 1, 4, 3, 5, 1, 3, respectively, and color $x_2x_3$ from $\{3,4\}$ with respect to 1 and $\sU_\phi(y_3)$.  This yields a good coloring of $G$.

If $(\psi(x_5y_5), \psi(x_7y_7)) = (4,5)$, suppose $\sU_\phi(y_6) \neq \{1,2,5\}$.  Then color $x_0x_1, x_1x_2, x_3x_4$, $x_4x_5$, $x_5x_6, x_6x_7, x_7x_0, x_4x_0$ with 3, 2, 4, 3, 5, 1, 5, 4, 1, respectively, and color $x_2x_3$ from $\{1,3\}$ with respect to 4 and $\sU_\phi(y_3)$.  This yields a good coloring of $G$.

So $\sU_\phi(y_6) = \{1,2,5\}$.  We then color $x_0x_1, x_1x_2, x_3x_4, x_4x_5, x_5x_6, x_6x_7, x_7x_0, x_4x_0$ with 4, 2, 1, 5, 3, 4, 1, 3, respectively, and color $x_2x_3$ from $\{3,4\}$ with respect to 1 and $\sU_\phi(y_3)$.  This yields a good coloring of $G$.

Thus, in each of the three cases, we obtain a good coloring of $G$.   This proves the lemma.
\end{proof}

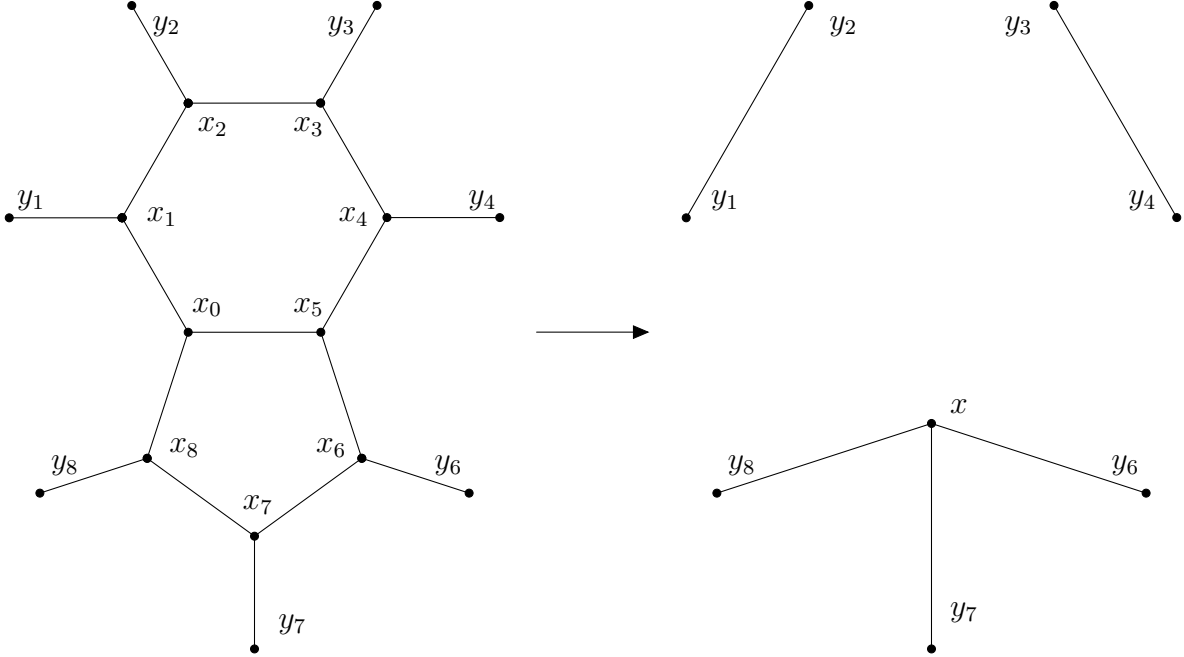
\begin{figure}[h]
\centering
\begin{tikzpicture}[line cap=round,line join=round,>=triangle 45,x=0.75cm,y=0.75cm]
\clip(-11,-6) rectangle (11,6);
\draw (-6.0,-3.62)-- (-6.0,-5.62);
\draw (-4.097886967409693,-2.238033988749895)-- (-2.195773934819386,-2.85606797749979);
\draw (-7.902113032590307,-2.238033988749895)-- (-9.804226065180615,-2.85606797749979);
\draw (-3.652480562338399,2.032090915285201)-- (-1.6524805623383987,2.032090915285201);
\draw (-4.828051066923345,4.064056926535307)-- (-3.8280510669233445,5.796107734104184);
\draw (-7.175570504584946,4.061966011250106)-- (-8.175570504584945,5.794016818818983);
\draw (-8.347519437661601,2.027909084714799)-- (-10.347519437661601,2.0279090847148002);
\draw (-7.175570504584946,4.061966011250106)-- (-4.828051066923345,4.064056926535307);
\draw (-4.828051066923345,4.064056926535307)-- (-3.652480562338399,2.032090915285201);
\draw (-3.652480562338399,2.032090915285201)-- (-4.824429495415053,-0.001966011250105426);
\draw (-4.824429495415053,-0.001966011250105426)-- (-7.175570504584946,-0.001966011250105204);
\draw (-7.175570504584946,-0.001966011250105204)-- (-8.347519437661601,2.027909084714799);
\draw (-8.347519437661601,2.027909084714799)-- (-7.175570504584946,4.061966011250106);
\draw (-7.175570504584946,-0.001966011250105204)-- (-7.902113032590307,-2.2380339887498946);
\draw (-4.824429495415053,-0.001966011250105426)-- (-4.097886967409693,-2.238033988749895);
\draw (-7.902113032590307,-2.2380339887498946)-- (-6.0,-3.62);
\draw (-6.0,-3.62)-- (-4.097886967409693,-2.238033988749895);
\draw (6, -1.62)--(2.196,-2.856);
\draw (6, -1.62)--(9.804,-2.856);
\draw (6, -1.62)--(6, -5.62);
\draw [->] (-1.0,0.0) -- (1.0,0.0);
\draw (-7.2,4) node[anchor=north west] {$x_2$};
\draw (-5.5,4) node[anchor=north west] {$x_3$};
\draw (-4.7,2.4) node[anchor=north west] {$x_4$};
\draw (-5.5,0.8) node[anchor=north west] {$x_5$};
\draw (-5.1,-1.7) node[anchor=north west] {$x_6$};
\draw (-6.4,-2.7) node[anchor=north west] {$x_7$};
\draw (-7.7,-1.7) node[anchor=north west] {$x_8$};
\draw (-7.3,0.8) node[anchor=north west] {$x_0$};
\draw (-8.1,2.4) node[anchor=north west] {$x_1$};
\draw (-8,5.8) node[anchor=north west] {$y_2$};
\draw (4,5.8) node[anchor=north west] {$y_2$};
\draw (-4.9,5.8) node[anchor=north west] {$y_3$};
\draw (3.824429495415055,5.794016818818983)--(1.6524805623383987,2.0279090847148002);
\draw (8.171948933076656,5.796107734104184)--(10.347519437661601,2.032090915285201);
\draw (7.1, 5.8) node[anchor=north west] {$y_3$};
\draw (-2.4,2.7) node[anchor=north west] {$y_4$};
\draw (-3,-2) node[anchor=north west] {$y_6$};
\draw (-5.767057190082635,-4.817012066115719) node[anchor=north west] {$y_7$};
\draw (-9.8,-2) node[anchor=north west] {$y_8$};
\draw (-10.4,2.7) node[anchor=north west] {$y_1$};
\draw (1.9,2.7) node[anchor=north west] {$y_1$};
\draw (2.2, -2) node[anchor=north west] {$y_8$};
\draw (9,-2) node[anchor=north west] {$y_6$};
\draw (9.3,2.7) node[anchor=north west] {$y_4$};
\draw (6.139342809917368,-4.599212066115719) node[anchor=north west] {$y_7$};
\draw (6.14, -1.) node[anchor=north west] {$x$};
\begin{scriptsize}
\draw [fill=black] (-6.0,-3.62) circle (1.5pt);
\draw [fill=black] (-4.097886967409693,-2.238033988749895) circle (1.5pt);
\draw [fill=black] (-4.824429495415053,-0.001966011250105426) circle (1.5pt);
\draw [fill=black] (-7.175570504584946,-0.001966011250105204) circle (1.5pt);
\draw [fill=black] (-7.902113032590307,-2.2380339887498946) circle (1.5pt);
\draw [fill=black] (-3.652480562338399,2.032090915285201) circle (1.5pt);
\draw [fill=black] (-4.828051066923345,4.064056926535307) circle (1.5pt);
\draw [fill=black] (-7.175570504584946,4.061966011250106) circle (1.5pt);
\draw [fill=black] (-8.347519437661601,2.027909084714799) circle (1.5pt);
\draw [fill=black] (-6.0,-5.62) circle (1.5pt);
\draw [fill=black] (-4.097886967409693,-2.238033988749895) circle (1.5pt);
\draw [fill=black] (-2.195773934819386,-2.85606797749979) circle (1.5pt);
\draw [fill=black] (-7.902113032590307,-2.238033988749895) circle (1.5pt);
\draw [fill=black] (-9.804226065180615,-2.85606797749979) circle (1.5pt);
\draw [fill=black] (-1.6524805623383987,2.032090915285201) circle (1.5pt);
\draw [fill=black] (-4.828051066923345,4.064056926535307) circle (1.5pt);
\draw [fill=black] (-3.8280510669233445,5.796107734104184) circle (1.5pt);
\draw [fill=black] (-7.175570504584946,4.061966011250106) circle (1.5pt);
\draw [fill=black] (-8.175570504584945,5.794016818818983) circle (1.5pt);
\draw [fill=black] (-7.175570504584946,4.061966011250106) circle (1.5pt);
\draw [fill=black] (-8.347519437661601,2.027909084714799) circle (1.5pt);
\draw [fill=black] (-10.347519437661601,2.0279090847148002) circle (1.5pt);
\draw [fill=black] (-8.347519437661601,2.027909084714799) circle (1.5pt);
\draw [fill=black] (-8.347519437661601,2.027909084714799) circle (1.5pt);
\draw [fill=black] (3.824429495415055,5.794016818818983) circle (1.5pt);
\draw [fill=black] (8.171948933076656,5.796107734104184) circle (1.5pt);
\draw [fill=black] (10.347519437661601,2.032090915285201) circle (1.5pt);
\draw [fill=black] (9.804226065180615,-2.85606797749979) circle (1.5pt);
\draw [fill=black] (6.0,-5.62) circle (1.5pt);
\draw [fill=black] (6.0,-1.62) circle (1.5pt);
\draw [fill=black] (2.195773934819386,-2.85606797749979) circle (1.5pt);
\draw [fill=black] (1.6524805623383987,2.0279090847148002) circle (1.5pt);
\end{scriptsize}
\end{tikzpicture}
\caption{Forming $G'$ from $G$}
\label{fig:No5-6face}
\end{figure}

\begin{lemma}\label{No5-6face}
No 5-face in $G$ can share an edge with a 6-face. 
\end{lemma}

\begin{proof}
Suppose that a 5-face and a 6-face share an edge.  By Lemmas \ref{NoTriangle} and \ref{No2on5cycle}, their boundaries form a 9-cycle, $x_0 x_1 \dots x_8$ so that $x_5x_0 \in E(G)$ .  By Lemmas \ref{No2on5cycle} and \ref{No2on6cycle}, each $x_i$ is a 3-vertex.   Additionally, Lemmas \ref{NoTriangle},  \ref{separating} and \ref{No4cycle}  imply that each $x_i$ other than $x_5, x_0$ has a third neighbor $y_i$ not on the 9-cycle.   By these same lemmas, the $y_i$'s are distinct from one another except possibly $y_7 \in \{y_2,y_3\}$, and furthermore $y_1y_2, y_3y_4 \notin E(G)$.  

Let $G'$ denote the graph obtained from $G$ by deleting $x_0, x_1, \dots, x_8$, adding a vertex $x$, and adding the edges $y_1y_2, y_3y_4, xy_6, xy_7, xy_8$ (see Figure \ref{fig:No5-6face}).  Observe that $G'$ is a subcubic planar multigraph with mutliplicity at most two, and so by the minimality of $G$, $G'$ has a good coloring.  We impose this coloring onto $G$ by coloring $x_1y_1$ and $x_2y_2$ with the color used on $y_1y_2$, coloring $x_3y_3$ and $x_4y_4$ with the color used on $y_3y_4$, and for $i \in \{6,7,8\}$, coloring $x_iy_i$ with the color used on $xy_i$.  Let $\psi$ denote this good partial coloring of $G$.

By the construction of $G'$, $|\\psi(x_iy_i): i \in \{6, 7, 8\}| = 3$.  So we can further extend our good partial coloring by coloring $x_6x_7$ and $x_7x_8$ with $\psi(x_8y_8)$ and $\psi(x_6y_6)$, respectively.  By Lemma \ref{extend}, we can further extend this by coloring $x_8x_0$ and $x_5x_6$.  Let $\phi$ denote this good partial coloring of $G$.  Thus, it only remains to color the edges of the 6-cycle, $x_0x_1x_2x_3x_4x_5$.

Without loss of generality, suppose $\phi(x_1y_1) = \phi(x_2y_2) = 1$, and let $\alpha := \phi(x_8x_0)$ and $\beta := \phi(x_5x_6)$.  Up to relabeling colors and symmetry, we may assume that $\phi(x_3y_3) = \phi(x_4y_4) \in \{1,2\}$.  Note that $\sU_\phi(x_8)\setminus\{\alpha\} = \sU_\phi(x_6)\setminus\{\beta\} = \{\phi(x_6y_6), \phi(x_8y_8)\}$.

Unlike in the proof of Lemma \ref{No5-5face}, we can always extend $\phi$ to a good coloring of $G$ by case analysis, the details of which can be found in the Appendix.  Thus we assume the lemma holds.
\end{proof}


\section{Proof of Theorem \ref{thm:conj}}\label{sec:proof}

We are now ready to prove Theorem \ref{thm:conj} via discharging using the lemmas from Sections \ref{struct1}, \ref{struct2} and \ref{struct3},

\begin{proof}
By Euler's formula,
\[
\sum_{v\in V(G)}(2d(v)-6)+\sum_{f\in F(G)}(d(f)-6)=-12.
\]

Thus, if we assign  to each vertex $v$ the initial charge $2d(v) - 6$ and  to each face $f$ the initial charge $d(f)-6$, 
then the overall charge will be $ - 12$.  We now redistribute charges among faces and vertices so
 that the final charge of every face and every vertex is nonnegative, a contradiction.

{\bf Discharging Rules:}\\
(R1) Every 2-vertex receives 1 from each incident face.\\
(R2) Every 5-face receives $\frac{1}{5}$ from each adjacent face. \\

By 
 Rule~(R1), at the end of discharging, each $2$-vertex 
 will have  charge $-2+1+1=0$. The charge of each 3-vertex does not change and remains $0$.

 By Rule~(R2) and Lemmas \ref{No2on5cycle} and \ref{No5-5face}, the final charge of every $5$-face is $5 - 6 + 5 \times \frac{1}{5} = 0$.

By Lemmas \ref{No2on6cycle} and \ref{No5-6face}, each 6-face gives no charge.  Thus, as it starts with zero charge and does not receives any charge, the final charge is zero.

By Lemmas \ref{No2on7face} and \ref{No5-5face}, each 7-face contains only 3-vertices and is adjacent to at most three 5-faces.  Thus, the final charge is at least $7 - 6 - 3 \times \frac{1}{5} = \frac{2}{5}$.

By Lemmas \ref{No5-5face} and \ref{distance>=5}, each $k$-face, $k \ge 8$, is adjacent to at most $\lfloor \frac{k}{2} \rfloor$ 5-faces 
and contains at most $\lfloor \frac{k}{5}\rfloor$ 2-vertices on its boundary.  
Thus, the final charge is at least  
$k - 6 - \left \lfloor \frac{k}{5} \right \rfloor \times 1 - \left \lfloor \frac{k}{2} \right \rfloor \times \frac{1}{5}$, which is positive for $k \ge 8$.

This completes the proof.
\end{proof}


\noindent\textbf{Future Questions.}  
Many analogous questions asked concerning the strong chromatic index of graphs can be asked regarding the $k$-intersection chromatic index as well (see \cite{FGST}).  As it pertains to the 2-intersection chromatic index of subcubic graphs, if a subcubic graph $G$ has the property that every 3-vertex is adjacent to only 2-vertices, then $\tchi(G) = \chi(G)$.  So we may assume that $G$ has  two adjacent 3-vertices.  This implies that $\tchi(G) \ge 4$.  This begs the following question: what type of graph must $G$ be to force $\tchi(G) = 4$?\\


\noindent {\bf Acknowledgment.}  The author wishes to thank Alexandr V. Kostochka and Thomas Mahoney for their helpful discussions.



\newpage
\section*{Appendix}

In this section we provide the detailed proofs of Lemmas \ref{No5-5face} and \ref{No5-6face}.

\begin{proof}[Proof of Lemma \ref{No5-5face}]
We assume $\psi$ and $\phi$ to be as described in the proof of Lemma \ref{No5-5face} in Section \ref{struct3}.  Thus, in order to extend $\phi$, it remains to color the edges of the cycle $x_0x_1x_2x_3x_4$.  As a result, when we `color the cycle in order' we color the edges $x_0x_1, x_1x_2, x_2x_3, x_3x_4, x_4x_0$ in this order.

We will break the following argument into cases depending on $\sU_\phi(x_7)\setminus\{\alpha\}$, and within each case we consider the values of $\alpha$ and $\beta$.  While each argument is relatively short, we will oftentimes state and prove claims to aid in the readability.

\setcounter{case}{0}
\begin{case}\label{5.5.1}
$\sU_\phi(x_7)\setminus\{\alpha\} = \{1,2\}$.
\end{case}

In this case we may assume without loss of generality that $\phi(x_2y_2) = 3$.

\begin{subcase}
$\alpha = \beta = 3$.
\end{subcase}

\begin{claim}
$\sU_\phi(y_3) = \{1,2,5\}$.
\end{claim}

\begin{proof}
Suppose $\sU_\phi(y_3) \neq \{1,2,5\}$. Also suppose $\sU_\phi(y_1) \neq \{1,2,4\}$.  Color $x_2x_3, x_3x_4, x_4x_0$ with 5, 1, 5, respectively.  Then color $x_1x_2$ (and $x_0x_1$) from $\{2,4\}$ with respect to 5 and $\sU_\phi(y_2)$.  This yields a good coloring of $G$.

So $\sU_\phi(y_1) = \{1,2,4\}$.  Suppose $\sU_\phi(y_3) \neq \{1,2,4\}$.  Color $x_2x_3, x_3x_4, x_4x_0$ with 4, 1, 4, respectively.  Then color $x_1x_2$ (and $x_0x_1$) from $\{2,5\}$ with respect to 4 and $\sU_\phi(y_2)$.  This yields a good coloring of $G$.

So $\sU_\phi(y_3) = \{1,2,4\}$.  Color $x_0x_1, x_1x_2, x_3x_4, x_4x_0$ with 2, 5, 5, 4, respectively, and then color $x_2x_3$ from $\{1,4\}$ with respect to 5 and $\sU_\phi(y_2)$.  This yields a good coloring of $G$.
\end{proof}

A similar argument shows that $\sU_\phi(y_3) = \{1,2,4\}$ by switching the roles of 4 and 5.  This contradicts the above claim and proves the subcase.

\begin{subcase}
$\alpha = \beta = 4$.
\end{subcase}

Suppose $\sU_\phi(y_3) \neq \{1,2,5\}$.   Also suppose $\sU_\phi(y_2) \neq \{3,4,5\}$.  Color $x_1x_2, x_2x_3, x_3x_4, x_4x_0$ with 4, 5, 1, 5, respectively, and color $x_0x_1$ from $\{2,3\}$ with respect to 4 and $\sU_\phi(y_1)$.  This yields a good coloring of $G$.  So $\sU_\phi(y_2) = \{3,4,5\}$.  Color $x_0x_1, x_2x_3, x_3x_4, x_4x_0$ with 2, 1, 5, 3, respectively, and color $x_1x_2$ from $\{4,5\}$ with respect to 2 and $\sU_\phi(y_1)$.  This yields a good coloring of $G$.

So $\sU_\phi(y_3) = \{1,2,5\}$, and by a symmetric argument $\sU_\phi(y_1) = \{1,2,5\}$.  We now color $x_0x_1, x_2x_3, x_3x_4, x_4x_0$ with 3, 4, 1, 5, respectively, and color $x_1x_2$ from $\{2,5\}$ with respect to 4 and $\sU_\phi(y_2)$.  This yields a good coloring of $G$.

\begin{subcase}
$(\alpha,\beta) = (3,4)$.
\end{subcase}

Suppose $\sU_\phi(y_3) \neq \{1,2,5\}$.  Color $x_1x_2, x_2x_3, x_3x_4, x_4x_0$ with 2, 1, 5, 1, respectively, and color$x_0x_1$ from $\{4,5\}$ with respect to 2 and $\sU_\phi(y_1)$.  This yields a good coloring of $G$.

So $\sU_\phi(y_3) = \{1,2,5\}$.  Suppose $\sU_\phi(y_2) \neq \{3,4,5\}$.  Color $x_1x_2, x_2x_3, x_3x_4, x_4x_0$ with 5, 4, 1, 5, respectively, and color $x_0x_1$ from $\{2,4\}$ with respect to 5 and $\sU_\phi(y_1)$.  This yields a good coloring of $G$.

So $\sU_\phi(y_2) = \{3,4,5\}$.  We then color $x_1x_2, x_2x_3, x_3x_4, x_4x_0$ with 2, 4, 5, 1, respectively, and color $x_0x_1$ from $\{4,5\}$ with respect to 2 and $\sU_\phi(y_1)$.  This yields a good coloring of $G$ and proves the subcase.

\begin{subcase}
$(\alpha,\beta) = (4,5)$.
\end{subcase}

\begin{claim}
$\sU_\phi(y_3) = \{1,2,4\}$, and by symmetry $\sU_\phi(y_1) = \{1,2,5\}$.
\end{claim}

\begin{proof}
Suppose $\sU_\phi(y_3) \neq \{1,2,4\}$.  Also suppose that $\sU_\phi(y_1) \neq \{1,2,5\}$.  Color $x_2x_3, x_3x_4, x_4x_0$ with 4, 1, 3, respectively, and color $x_1x_2$ (and $x_0x_1$) from $\{2,5\}$ with respect to 4 and $\sU_\phi(y_2)$.  This yields a good coloring of $G$.

So $\sU_\phi(y_1) = \{1,2,5\}$.  Now $\sU_\phi(y_2) = \{1,3,4\}$, otherwise color the cycle in order with 2, 4, 1, 4, 3.  This is a good coloring of $G$.  We then color $x_0x_1, x_1x_2, x_2x_3, x_4x_0$ with 2, 4, 5, 3, respectively, and color $x_3x_4$ from $\{1,4\}$ with respect to 5 and $\sU_\phi(y_3)$.  This yields a good coloring of $G$ and proves the claim.
\end{proof}

Now $\sU_\phi(y_2) = \{3,4,5\}$, otherwise color the cycle in order with 3, 5, 4, 3, 1.  We then color the cycle in order with 5, 4, 1, 3, 2.  This is a good coloring of $G$ and proves the subcase.

Up to relabeling the colors and symmetry, this completes all subcases and proves Case \ref{5.5.1}

\begin{case}\label{5.5.2}
$\sU_\phi(x_7)\setminus\{\alpha\} = \{1,3\}$.
\end{case}

In this case, $\alpha, \beta \in \{2,4,5\}$.   By the construction of $G'$, $\phi(x_2y_2) \in \{3,4,5\}$.  Up to relabeling, we may assume that either $\phi(x_2y_2) = 3$ or $\phi(x_2y_2) = 5$.

\begin{subcase}\label{(1,3).5}
$\phi(x_2y_2) = 5$.
\end{subcase}

\begin{subsubcase}
$\alpha = \beta$.
\end{subsubcase}

Let $\oalpha \in \{4,5\}\setminus\{\alpha\}$.

Suppose $\sU_\phi(y_3) \neq \{1,2,3\}$.  Also suppose $\sU_\phi(y_1) \neq \{1,2,4\}$.  If $\alpha = 2$, color the cycle in order with 4, 2, 1, 3, 5.  If $\alpha \in \{4,5\}$, color $x_0x_1, x_1x_2, x_4x_0$ with 2, 4, $\oalpha$, respectively.  We then color $x_2x_3$ (and $x_3x_4$) from $\{1,3\}$ with respect to 4 and $\sU_\phi(y_2)$.  This yields a good coloring of $G$.  So $\sU_\phi(y_1) = \{1,2,4\}$.  Color $x_0x_1, x_2x_3, x_3x_4, x_4x_0$ with 3, 3, 1, $\oalpha$, respectively, and color $x_1x_2$ from $\{2,4\}$ with respect to 3 and $\sU_\phi(y_2)$.  This yields a good coloring of $G$.

So $\sU_\phi(y_3) = \{1,2,3\}$.  Suppose $\sU_\phi(y_1) \neq \{1,2,3\}$.  If $\alpha = 2$, color the cycle in order with 3, 2, 1, 4, 5.  If $\alpha \in \{4,5\}$, color $x_2x_3, x_3x_4, x_4x_0$ with 4, 1, $\oalpha$, respectively.  We then color $x_1x_2$ (and $x_0x_1$) from $\{2,3\}$ with respect to 4 and $\sU_\phi(y_2)$.  This yields a good coloring of $G$.

So $\sU_\phi(y_1) = \{1,2,3\}$.  We now color $x_0x_1$ and $x_1x_2$ with 3 and 4, respectively.  If $\alpha = 2$, we color $x_4x_5$ and $x_5x_0$ with 5 and 4, respectively.  Otherwise, we color $x_4x_5$ and $x_5x_0$ with $\oalpha$ and 2, respectively.  In both cases we color $x_2x_3$ from $\{1,3\}$ with respect to 4 and $\sU_\phi(y_2)$.   This yields a good coloring of $G$ and proves the subcase.\\

Since $\alpha \neq \beta$, there exists $\gamma$ such that $\{\alpha, \beta, \gamma\} = \{2,4,5\}$.  We now show that in the remaining subcases, we may assume that $\sU_\phi(y_3) = \{1,2,3\}$.  If $\sU_\phi(y_3) \neq \{1,2,3\}$, color $x_0x_1$ and $x_4x_0$ with 3 and $\gamma$, respectively.  Then color $x_1x_2$ from $\{2,4\}$ with respect to 3 and $\sU_\phi(y_1)$.  Let $\delta$ denote the color used on $x_1x_2$.  We then color $x_2x_3$ (and $x_3x_4$) from $\{1,3\}$ with respect to $\delta$ and $\sU_\phi(y_2)$.  This yields a good coloring of $G$.  So $\sU_\phi(y_3) = \{1,2,3\}$, as desired.

\begin{subsubcase}
$\alpha \neq 2$.
\end{subsubcase}

Suppose $\sU_\phi(y_1) \neq \{1,3,4\}$.  Color $x_0x_1, x_1x_2, x_3x_4, x_4x_0$ with 3, 4, $\alpha$, $\gamma$, respectively, and color $x_2x_3$ from $\{1,3\}$ with respect to 4 and $\sU_\phi(y_2)$.  This yields a good coloring of $G$.

So $\sU_\phi(y_1) = \{1,3,4\}$.  If $\alpha = 4$, color the cycle in order with 3, 2, 1, $\alpha$, $\gamma$.  This is a good coloring of $G$.  So $\alpha = 5$ and $\{\beta, \gamma\} = \{2,4\}$.  If $\beta = 2$, color the cycle in order with 4, 2, 1, 4, 3.  This is a good coloring of $G$.  

Thus, $\alpha = 5$ and $\beta = 4$.  Color $x_0x_1, x_1x_2, x_3x_4, x_4x_0$ with 2, 4, 5, 1, respectively, and color $x_2x_3$ from $\{1,3\}$ with respect to 4 and $\sU_\phi(y_2)$. This yields a good coloring of $G$.

\begin{subsubcase}
$\alpha = 2$.
\end{subsubcase}

Here $\beta \in \{4,5\}$.  Suppose first that $\beta = 4$.  Also suppose $\sU_\phi(y_1) \neq \{1,3,4\}$.  Color $x_0x_1, x_1x_2, x_3x_4, x_4x_0$ with 4, 3, 5, 3, respectively, and color $x_2x_3$ from $\{1,4\}$ with respect to 3 and $\sU_\phi(y_2)$.  This yields a good coloring of $G$.

So $\sU_\phi(y_1) = \{1,3,4\}$.  $\sU_\phi(y_2) = \{2,4,5\}$, otherwise color the cycle in order with 4, 2, 4, 3, 5.  We then color the cycle in order with 5, 3, 4, 5, 1.  This is a good coloring of $G$.

So $\beta = 5$.  Now $\sU_\phi(y_1) = \{1,2,4\}$, otherwise we color the cycle in order with 4, 2, 1, 4, 3.  We then color $x_0x_1, x_2x_3, x_3x_4, x_4x_0$ with 5, 4, 3, 4, respectively, and color $x_1x_2$ from $\{2,3\}$ with respect to 4 and $\sU_\phi(y_2)$.  This yields a good coloring of $G$.\\

This completes the proof of Subcase \ref{(1,3).5}

\begin{subcase}\label{(1,3).3}
$\phi(x_2y_2) = 3$.
\end{subcase}

\begin{subsubcase}
$\alpha = \beta = 2$.
\end{subsubcase}

Suppose $\sU_\phi(y_3) \neq \{1,2,4\}$.  Color $x_2x_3, x_3x_4, x_4x_0$ with 4, 1, 5, respectively.  We then color $x_1x_2$ from $\{2,5\}$ with respect to 4 and $\sU_\phi(y_2)$.  Let $\gamma$ denote the color used on $x_1x_2$, and color $x_0x_1$ from $\{3,4\}$ with respect to $\gamma$ and $\sU_\phi(y_1)$.  This yields a good coloring of $G$.

So $\sU_\phi(y_3) = \{1,2,4\}$.  Suppose $\sU_\phi(y_1) \neq \{1,4,5\}$.  Color $x_0x_1, x_1x_2, x_3x_4, x_4x_0$ with 4, 5, 3, 5, respectively, and color $x_2x_3$ from $\{1,4\}$ with respect to 5 and $\sU_\phi(y_2)$.  This yields a good coloring of $G$.

So $\sU_\phi(y_1) = \{1,4,5\}$.  We then color the cycle in order with 4, 2, 1, 5, 3.  This is a good coloring of $G$.

\begin{subsubcase}\label{(1,3)(4,4)}
$\alpha = \beta \in \{4,5\}$.
\end{subsubcase}

Let $\oalpha$ such that $\{\alpha, \oalpha\} = \{4,5\}$.  Suppose $\sU_\phi(y_3) \neq \{1,2,\oalpha\}$.  Color $x_2x_3, x_3x_4, x_4x_0$ with $\oalpha$, 1, 2, respectively.  Color $x_1x_2$ from $\{2,\alpha\}$ with respect to $\oalpha$ and $\sU_\phi(y_2)$.  Let $\gamma$ denote the color used on $x_1x_2$, and color $x_0x_1$ from $\{3,\oalpha\}$ with respect to $\gamma$ and $\sU_\phi(y_1)$.   This yields a good coloring of $G$.

So $\sU_\phi(y_3) = \{1,2,5\}$.  Suppose $\sU_\phi(y_1) \neq \{1,4,5\}$.  Now color $x_0x_1, x_1x_2, x_3x_4, x_4x_0$ with $\oalpha$, $\alpha$, 3, 2, respectively, and color $x_2x_3$ from $\{1,\oalpha\}$ with respect to $\alpha$ and $\sU_\phi(y_2)$.  This yields a good coloring of $G$.

So $\sU_\phi(y_1) = \{1,4,5\}$.  Now $\sU_\phi(y_2) = \{3,4,5\}$, otherwise color the cycle in order with 2, 5, 4, 5, 3.  We then color the cycle in order with 2, 4, 1, 3, 5.  This is a good coloring of $G$.

\begin{subsubcase}
$(\alpha,\beta) = (2,4)$.
\end{subsubcase}

Suppose $\sU_\phi(y_3) \neq \{1,2,4\}$.  Color $x_2x_3, x_3x_4, x_4x_0$ with 4, 1, 5, respectively.  Now color $x_1x_2$ from $\{2,5\}$ with respect to 4 and $\sU_\phi(y_2)$.  Let $\gamma$ denote the color used on $x_1x_2$, and color $x_0x_1$ from $\{3,4\}$ with respect to $\gamma$ and $\sU_\phi(y_1)$.  This yields a good coloring of $G$.

So $\sU_\phi(y_3) = \{1,2,4\}$.  Suppose $\sU_\phi(y_1) \neq \{1,4,5\}$.  Color $x_0x_1, x_1x_2, x_3x_4, x_4x_0$ with 4, 5, 5, 3, respectively, and $x_2x_3$ from $\{1,4\}$ with respect to 5 and $\sU_\phi(y_2)$.  This yields a good coloring of $G$.

So $\sU_\phi(y_1) = \{1,4,5\}$.  We then color the cycle in order with 5, 2, 1, 5, 3.  This is a good coloring of $G$.

\begin{subsubcase}
$(\alpha,\beta) = (4,2)$.
\end{subsubcase}

Suppose $\sU_\phi(y_3) \neq \{1,2,5\}$.  Now $\sU_\phi(y_1) = \{1,2,5\}$, otherwise color the cycle in order wth 5, 2, 1, 5, 3.  $\sU_\phi(y_2) = \{1,3,4\}$, otherwise color the cycle in order with 2, 4, 1, 5, 3.  $\sU_\phi(y_3) = \{1,2,4\}$, otherwise color the cycle in order with 3, 2, 4, 1, 5.  We then color the cycle in order with 5, 4, 5, 4, 3.  These are good colorings of $G$.

So $\sU_\phi(y_3) = \{1,2,5\}$.  Suppose $\sU_\phi(y_2) \neq \{3,4,5\}$.  Color $x_1x_2, x_2x_3, x_3x_4, x_4x_0$ with 5, 4, 1, 5, respectively, and color $x_0x_1$ from $\{2,3\}$ with respect to 5 and $\sU_\phi(y_1)$. This yields a good coloring of $G$.  

So $\sU_\phi(y_2) = \{3,4,5\}$.  By the construction of $G$, $\sU_\phi(y_1) \neq \{1,2,3\}$.  So we color the cycle in order with 3, 2, 4, 1, 5.  This is a good coloring of $G$.

\begin{subsubcase}\label{(1,3)(4,5)}
$(\alpha,\beta) = (4,5)$.
\end{subsubcase}

\begin{claim}
$\sU_\phi(y_1) = \{1,4,5\}$, $\sU_\phi(y_2) = \{3,4,5\}$, and $\sU_\phi(y_3) = \{1,2,4\}$.
\end{claim}

\begin{proof}
Suppose first that $\sU_\phi(y_3) \neq \{1,2,4\}$.   $\sU_\phi(y_2) = \{2,3,4\}$, otherwise color $x_1x_2, x_2x_3$, $x_3x_4, x_4x_0$ with 2, 4, 1, 2, respectively, and color $x_0x_1$ from $\{3,5\}$ with respect to 2 and $\sU_\phi(y_1)$.  This yields a good coloring of $G$.  Now $\sU_\phi(y_1) = \{1,3,5\}$, otherwise color the cycle in order with 3, 5, 4, 1, 2.  We then color the cycle in order with 2, 5, 1, 4, 3.  This is a good coloring of $G$.

Thus, $\sU_\phi(y_3) = \{1,2,4\}$.  Now suppose $\sU_\phi(y_2) \neq \{3,4,5\}$.  Then $\sU_\phi(y_1) = \{1,2,4\}$, otherwise color the cycle in order with 2, 4, 5, 4, 3.  We then color the cycle in order with 3, 5, 4, 3, 2.  This is a good coloring of $G$.

Thus, $\sU_\phi(y_2) = \{3,4,5\}$.  Then $\sU_\phi(y_1) = \{1,4,5\}$, otherwise color 5, 4, 1, 3, 2.  This is a good coloring of $G$ and proves the claim.
\end{proof}

To complete this subcase, we will reconsider the good partial coloring of $G$ $\psi$.  Recall that $\{\phi(x_5y_5), \phi(x_7y_7)\} = \{1,3\}$.  So by the construction of $G'$,  $\phi(x_6y_6) \in  \{2,4,5\}$.  Since $\alpha = \phi(x_7x_0) = 4$ and $\beta = \phi(x_4x_5) = 5$, $\phi(x_6y_6) = 2$.  If $\sU_\phi(y_5) \neq \{1,3,4\}$, we could recolor $x_4x_5$ with 4 and proceed as in Subcase \ref{(1,3)(4,4)}.  Similarly, $\sU_\phi(y_7) = \{1,3,5\}$.  

Recall that under $\psi$ the edges of the cycle $x_0x_1\dots x_7$ along with the edge $x_4x_0$ are the remaining uncolored edges.  Thus, when we `color the cycle in order' we color the edges $x_0x_1, x_1x_2, \dots, x_6x_7, x_7x_0$ in this order.  

Suppose $\sU_\psi(y_6) \neq \{2,3,5\}$.  If $(\psi(x_5y_5), \psi(x_7y_7)) = (1,3)$, color $x_4x_0$ with 1 and color the cycle in order with 3, 2, 5, 4, 5, 3, 5, 4.  If $(\psi(x_5y_5), \psi(x_7y_7)) = (3,1)$, color $x_4x_0$ with 2 and color the cycle in order with 3, 2, 5, 4, 1, 5, 3, 4.  In either case, this is a good coloring of $G$.

So $\sU_\psi(y_6) = \{2,3,5\}$.  If $(\psi(x_5y_5), \psi(x_7y_7)) = (1,3)$, color $x_4x_0$ with 1 and color the cycle in order with 2, 5, 1, 3, 5, 4, 1, 4.  If $(\psi(x_5y_5), \psi(x_7y_7)) = (3,1)$, color $x_4x_0$ with 4 and color the cycle in order with 2, 5, 1, 3, 5, 1, 4, 3.  In either case, this is a good coloring of $G$.\\

Up to relabeling the colors and symmetry, this completes the proof of Subcase \ref{(1,3).3}, and so completes the proof of Case \ref{5.5.2}.

\begin{case}\label{5.5.3}
$\sU_\phi(x_7)\setminus\{\alpha\} = \{4,5\}$.
\end{case}

In this case, $\alpha, \beta \in \{1,2,3\}$.   By the construction of $G$, $\phi(x_2y_2) \in \{3,4,5\}$.  Up to relabeling, we may assume that $\phi(x_2y_2) \in \{3,5\}$.

\begin{subcase}\label{(4,5).3}
$\phi(x_2y_2) = 3$.
\end{subcase}

\begin{subsubcase}
$\alpha = \beta$.
\end{subsubcase}

\begin{claim}
$\sU_\phi(y_3) = \{2,4,5\}$.
\end{claim}

\begin{proof}
Let $\{\alpha_1, \alpha_2, \alpha_3\} = \{1,2,3\}$ so that without loss of generaltiy, $\alpha = \alpha_1$.  Suppose $\sU_\phi(y_3) \neq \{2,4,5\}$.  Also suppose $\sU_\phi(y_2) \neq \{3,4,5\}$.  Color $x_0x_1$ from $\{alpha_2,\alpha_3\}\setminus\{1\}$.  Without loss of generaltiy, assume $x_0x_1$ is colored with $\alpha_2$.  We then color $x_4x_0$ with $\alpha_3$.  Color $x_1x_2, x_3x_4$ (and $x_2x_3$) from $\{4,5\}$ with respect to $\alpha_2$ and $\sU_\phi(y_1)$.  This yields a good coloring of $G$.

So $\sU_\phi(y_2) = \{3,4,5\}$.  Color $x_1x_2$ with 2.  If $\alpha_1  = 3$, color $x_4x_0$ with $\alpha_2$.  Otherwise, color $x_4x_0$ with 3.  We then color $x_0x_1, x_2x_3$ (and $x_3x_4$) from $\{4,5\}$ with respect to 2 and $\sU_\phi(y_1)$.  This yields a good coloring of $G$ and proves the claim.
\end{proof}

We now color $x_1x_2$ and $x_2x_3$ with 2 and 1, respectively.  If $\alpha_1 = 3$, color $x_4x_0$ with $\alpha_2$, and color $x_0x_1$ (and $x_3x_4$) from $\{4,5\}$ with respect to 2 and $\sU_\phi(y_1)$.  Otherwise, color $x_4x_0$ with 3, and color $x_0x_1$ (and $x_3x_4$) from $\{4,5\}$ with respect to 2 and $\sU_\phi(y_1)$.  These are good colorings of $G$.

\begin{subsubcase}
$\{\alpha,\beta\} = \{1,2\}$.  
\end{subsubcase}

Color $x_1x_2, x_2x_3, x_4x_0$ with 2, 1, 3, respectively.  We then color $x_0x_1$ from $\{4,5\}$ with respect to 2 and $\sU_\phi(y_1)$, and color $x_3x_4$ from $\{4,5\}$ with respect to 1 and $\sU_\phi(y_1)$.  This yields a good coloring of $G$.

\begin{subsubcase}
$(\alpha,\beta) = (2,3)$.
\end{subsubcase}

\begin{claim}
$\sU_\phi(y_3) = \{2,4,5\}$.
\end{claim}

\begin{proof}
Suppose $\sU_\phi(y_3) \neq \{2,4,5\}$.  Also suppose $\sU_\phi(y_2) \neq \{3,4,5\}$.  Color $x_0x_1$ and $x_4x_0$ with 3 and 1, respectively.  Now color $x_1x_2, x_3x_4$ (and $x_2x_3$) from $\{4,5\}$ with respect to 3 and $\sU_\phi(y_1)$.  This yields a good coloring of $G$.

So $\sU_\phi(y_2) = \{3,4,5\}$.  Suppose $\sU_\phi(y_1) \neq \{1,4,5\}$.  Color $x_0x_1, x_1x_2, x_2x_3, x_4x_0$ with 4, 5, 1, 1, respectively, and color $x_3x_4$ from $\{4,5\}$ with respect to 1 and $\sU_\phi(y_3)$.  This yields a good coloring of $G$.

So $\sU_\phi(y_1) = \{1,4,5\}$.  Color $x_0x_1, x_1x_2, x_3x_4$ with 3, 2, 1, respectively.  We then color $x_2x_3$ (and $x_4x_0$) from $\{4,5\}$ with respect to 1 and $\sU_\phi(y_3)$.  This yields a good coloring of $G$ and proves the claim.
\end{proof}

Suppose $\sU_\phi(y_1) \neq \{1,4,5\}$.  Color $x_2x_3, x_3x_4, x_4x_0$ with 1, 4, 1, respectively.  We then color $x_1x_2$ (and $x_0x_1$) from $\{4,5\}$ with respect to 1 and $\sU_\phi(y_2)$.  This yields a good coloring of $G$.

So $\sU_\phi(y_1) = \{1,4,5\}$.  Color $x_0x_1, x_1x_2, x_3x_4, x_4x_0$ with 3, 2, 1, 4, respectively, and color $x_2x_3$ from $\{4,5\}$ with respect to 2 and $\sU_\phi(y_2)$.  This yields a good coloring of $G$.

\begin{subsubcase}
$(\alpha,\beta) = (1,3)$.
\end{subsubcase}

\begin{claim}
$\sU_\phi(y_1) = \{1,4,5\}$ and $\sU_\phi(y_2) = \{3,4,5\}$.
\end{claim}

\begin{proof}
Suppose first that $\sU_\phi(y_1) \neq \{1,4,5\}$.  Also suppose $\sU_\phi(y_2) \neq \{3,4,5\}$.  Color $x_3x_4$ and $x_4x_0$ with 1 and 2, respectively.  We then color $x_2x_3, x_0x_1$ (and $x_1x_2$) from $\{4,5\}$ with respect to 1 and $\sU_\phi(y_3)$.  This yields a good coloring of $G$.

So $\sU_\phi(y_2) = \{3,4,5\}$.  Color $x_0x_1, x_1x_2, x_2x_3, x_4x_0$ with 4, 5, 1, 2, respectively, and $x_3x_4$ from $\{4,5\}$ with respect to 1 and $\sU_\phi(y_3)$.  This yields a good coloring of $G$.

Thus, $\sU_\phi(y_1) = \{1,4,5\}$.  Now suppose $\sU_\phi(y_2) \neq \{3,4,5\}$.  Then $\sU_\phi(y_3) = \{2,4,5\}$, otherwise color the cycle in order with 3, 4, 5, 4, 2.  We then color the cycle in order with 2, 4, 5, 1, 5.  This is a good coloring of $G$ and proves the claim.
\end{proof}

To complete this subcase, we will reconsider the good partial coloring of $G$ $\psi$.   In a manner similar to that in Subcase \ref{(1,3)(4,5)}, we deduce that $\sU_\phi(y_5) = \{1,4,5\}$, $\sU_\phi(y_7) = \{3,4,5\}$, and $\phi(x_6y_6) = 2$.  We now recolor the edges of the cycle $x_0x_1 \dots x_7$ and the edge $x_4x_0$.

If $(\psi(x_5y_5), \psi(x_7y_7)) = (5,4)$, suppose $\sU_\psi(y_6) \neq \{2,3,5\}$.  Then color $x_0x_1, x_1x_2$, $x_3x_4, x_4x_5$, $x_5x_6$, $x_6x_7$, $x_7x_0, x_4x_0$ with 5, 2, 1, 4, 3, 5, 1, 3, respectively, and color $x_2x_3$ from $\{4,5\}$ with respect to 1 and $\sU_\phi(y_3)$.  This yields a good coloring of $G$.  

So $\sU_\phi(y_6) = \{2,3,5\}$.  We then color $x_0x_1, x_1x_2, x_3x_4, x_4x_5, x_5x_6, x_6x_7, x_7x_0, x_4x_0$ with 4, 2, 1, 3, 4, 1, 5, 2, respectively, and color $x_2x_3$ from $\{4,5\}$ with respect to 1 and $\sU_\phi(y_3)$.  This yields a good coloring of $G$.

A similar argument holds for $(\psi(x_5y_5), \psi(x_7y_7)) = (4,5)$ when considering whether or not $\sU_\phi(y_6)$ is $\{2,3,4\}$ by switching the roles of 4 and 5.\\

Up to relabeling the colors and symmetry, this completes the proof of Subcase \ref{(4,5).3}.

\begin{subcase}\label{(4,5).5}
$\phi(x_2y_2) = 5$.
\end{subcase}

\begin{subsubcase}
$\alpha = \beta = 1$.
\end{subsubcase}

Suppose $\sU_\phi(y_3) \neq \{2,3,4\}$.  Also suppose $\sU_\phi(y_2) \neq \{3,4,5\}$.  Color $x_0x_1$ and $x_4x_0$ with 5 and 2, respectively.  We then color $x_1x_2, x_3x_4$ (and $x_2x_3$) from $\{3,4\}$ with respect to 5 and $\sU_\phi(y_1)$.  This yields a good coloring of $G$.   So $\sU_\phi(y_2) = \{3,4,5\}$.  Now $\sU_\phi(y_1) = \{1,2,3\}$, otherwise color the cycle in order with 2, 3, 1, 3, 4.  This is a good coloring of $G$.  By the construction of $G'$, $\sU_\phi(y_3) \neq \{1,2,5\}$.  So we color the cycle in order with 2, 4, 1, 5, 3.  This is a good coloring of $G$.

So $\sU_\phi(y_3) = \{2,3,4\}$.  We now color $x_0x_1, x_2x_3, x_3x_4, x_4x_0$ with 2, 1, 3, 5, respectively.  If $2 \notin \sU_\phi(y_1)$, we color $x_1x_2$ from $\{3,4\}$ with respect to 1 and $\sU_\phi(y_2)$.  This yields a good coloring of $G$.

So $2 \in \sU_\phi(y_1)$, and by a similar argument $1 \in \sU_\phi(y_2)$.  We then color the cycle in order with 3, 4, 3, 5, 2.  This is a good coloring of $G$.\\

A symmetric argument holds when $\alpha = \beta = 2$.

\begin{subsubcase}
$\alpha = \beta = 3$.
\end{subsubcase}

Suppose $\sU_\phi(y_3) \neq \{1,2,4\}$.  Color $x_2x_3, x_3x_4, x_4x_0$ with 4, 1, 2, respectively.  We then color $x_1x_2$ from $\{2,3\}$ with respect to 4 and $\sU_\phi(y_2)$.  Let $\gamma$ denote the color used on $x_1x_2$, and color $x_0x_1$ from $\{4,5\}$ with respect to $\gamma$ and $\sU_\phi(y_1)$.  This yields a good coloring of $G$.

So $\sU_\phi(y_3) = \{1,2,4\}$.  Suppose $\sU_\phi(y_1) \neq \{1,3,4\}$.  Color $x_0x_1, x_1x_2, x_3x_4, x_4x_0$ with 4, 3, 5, 2, respectively, and color $x_2x_3$ from $\{1,4\}$ with respect to 3 and $\sU_\phi(y_2)$.  This yields a good coloring of $G$.

So $\sU_\phi(y_1) = \{1,3,4\}$.  Now $\sU_\phi(y_2) = \{3,4,5\}$, otherwise color the cycle in order with 2, 4, 3, 1, 5.  We then color the cycle in order with 5, 2, 3, 4, 1.  This is a good coloring of $G$.

\begin{subsubcase}
$(\alpha,\beta) = (1,2)$.
\end{subsubcase}

\begin{claim}
$\sU_\phi(y_3) = \{1,2,4\}$.
\end{claim}

\begin{proof}
Suppose $\sU_\phi(y_3) \neq \{1,2,4\}$.  Also suppose $\sU_\phi(y_1) \neq \{1,2,4\}$.  Color $x_0x_1, x_1x_2, x_4x_0$ with 4, 2, 3, respectively.  We then color $x_2x_3$ (and $x_3x_4$) from $\{1,4\}$ with respect to 2 and $\sU_\phi(y_2)$.  This yields a good coloring of $G$.

So $\sU_\phi(y_1) = \{1,2,4\}$.  Then $\sU_\phi(y_2) = \{2,4,5\}$, otherwise color the cycle in order with 3, 2, 4, 1, 5.  $\sU_\phi(y_3) = \{2,3,4\}$, otherwise color the cycle in order with 2, 3, 4, 3, 5.  We then color the cycle in order with 3, 2, 1, 3, 4.  These are good colorings of $G$ and prove the claim.
\end{proof}

Suppose $\sU_\phi(y_1) \neq \{1,2,3\}$.  Color $x_2x_3, x_3x_4, x_4x_0$ with 4, 3, 5, respectively. We then color $x_1x_2$ (and $x_0x_1$ from $\{2,3\}$ with respect to 4 and $\sU_\phi(y_2)$.  This yields a good coloring of $G$.  

So $\sU_\phi(y_1) = \{1,2,3\}$.  Color $x_0x_1, x_1x_2, x_4x_0$ with 3, 4, 5, respectively.  We then color $x_2x_3$ (and $x_3x_4$) from $\{1,3\}$ with respect to 4 and $\sU_\phi(y_2)$.  This yields a good coloring of $G$.

\begin{subsubcase}
$(\alpha,\beta) = (2,1)$.
\end{subsubcase}

\begin{claim}
$\sU_\phi(y_3) = \{2,3,4\}$.
\end{claim}

\begin{proof}
Suppose $\sU_\phi(y_3) \neq \{2,3,4\}$.  Also suppose $\sU_\phi(y_2) \neq \{2,3,5\}$.  Color $x_1x_2, x_2x_3, x_3x_4, x_4x_0$ with 2, 3, 4, 3, respectively, and color $x_0x_1$ from $\{4,5\}$ with respect to 2 and $\sU_\phi(y_1)$.  This yields a good coloring of $G$.

So $\sU_\phi(y_2) = \{2,3,5\}$.  Also $\sU_\phi(y_1) = \{1,4,5\}$, otherwise color the cycle in order with 5, 4, 3, 4, 3.  We then color the cycle in order with 3, 2, 4, 3, 4.  These are good colorings of $G$ and prove the claim.
\end{proof}

Color $x_1x_2, x_2x_3$ with 2 and 1, respectively.  Color $x_0x_1$ from $\{3,4\}$ with respect to 2 and $\sU_\phi(y_1)$.  If $x_0x_1$ is colored with 3, color $x_3x_4$ and $x_4x_0$ with 3 and 5, respectively.  If $x_0x_1$ is colored with 4, color $x_3x_4$ and $x_4x_0$ with 4 and 3, respectively.  In either case, this is a good coloring of $G$.

\begin{subsubcase}
$(\alpha,\beta) = (3,1)$.
\end{subsubcase}

Suppose $\sU_\phi(y_3) \neq \{2,3,4\}$.  Color both $x_1x_2$ and $x_4x_0$ with 2. Now color $x_0x_1$ from $\{4,5\}$ with respect to 2 and $\sU_\phi(y_1)$.  We then color $x_2x_3$ (and $x_3x_4$) from $\{3,4\}$ with respect to 2 and $\sU_\phi(y_2)$.  This yields a good coloring of $G$.

Color $x_0x_1, x_2x_3, x_3x_4, x_4x_0$ with 2, 1, 3, 4, respectively.  If $2 \notin \sU_\phi(y_1)$, color $x_1x_2$ from $\{3,4\}$ with respect to 1 and $\sU_\phi(y_2)$.  This yields a good coloring of $G$.

So $2 \in \sU_\phi(y_1)$, and by a similar argument $1 \in \sU_\phi(y_2)$.  We then color the cycle in order with 4, 3, 4, 5, 2.  This is a good coloring of $G$.

\begin{subsubcase}
$(\alpha,\beta) = (1,3)$.
\end{subsubcase}

\begin{claim}
$\sU_\phi(y_1) = \{1,3,4\}$ and $\sU_\phi(y_2) = \{3,4,5\}$.
\end{claim}

\begin{proof}
First suppose $\sU_\phi(y_1) \neq \{1,3,4\}$.  Color $x_2x_3$ and $x_4x_0$ with 1 and 2, respectively.  Now color $x_3x_4$ from $\{4,5\}$ with respect to 1 and $\sU_\phi(y_3)$.  We then color $x_1x_2$ (and $x_0x_1$) from $\{3,4\}$ with respect to 1 and $\sU_\phi(y_2)$.  This yields a good coloring of $G$.

So $\sU_\phi(y_1) = \{1,3,4\}$.  Now suppose $\sU_\phi(y_2) \neq \{3,4,5\}$.  Color $x_0x_1, x_1x_2, x_2x_3, x_4x_0$ with 5, 3, 4, 2, respectively, and $x_3x_4$ from $\{1,5\}$ with respect to 4 and $\sU_\phi(y_3)$.  This is a good coloring of $G$ and proves the claim.
\end{proof}

To complete this subcase, we will reconsider the good partial coloring of $G$ $\psi$.  In a manner similar to that in Subcase \ref{(1,3)(4,5)}, we deduce that $\sU_\phi(y_5) = \{1,4,5\}, \sU_\phi(y_7) = \{3,4,5\}$, and $\phi(x_6y_6) = 2$.  We now recolor the edges of the cycle $x_0x_1\dots x_7$ and the edge $x_4x_0$.

 If $(\psi(x_5y_5), \psi(x_7y_7)) = (5,4)$, suppose $\sU_\psi(y_6) \neq \{1,2,4\}$.  Then color $x_0x_1, x_1x_2$, $x_3x_4$, $x_4x_5$, $x_5x_6, x_6x_7, x_7x_0, x_4x_0$ with 3, 2, 4, 3, 4, 1, 5, 1, respectively, and color $x_2x_3$ from $\{1,3\}$ with respect to 4 and $\sU_\phi(y_3)$.  This yields a good coloring of $G$.

So $\sU_\phi(y_6) = \{1,2,4\}$.  We then color $x_0x_1, x_1x_2, x_3x_4, x_4x_5, x_5x_6, x_6x_7, x_7x_0, x_4x_0$ with 5, 2, 1, 4, 3, 5, 1, 3, respectively, and color $x_2x_3$ from $\{3,4\}$ with respect to 1 and $\sU_\phi(y_3)$.  This yields a good coloring of $G$.

If $(\psi(x_5y_5), \psi(x_7y_7)) = (4,5)$, suppose $\sU_\phi(y_6) \neq \{1,2,5\}$.  Then color $x_0x_1, x_1x_2, x_3x_4$, $x_4x_5$, $x_5x_6, x_6x_7, x_7x_0, x_4x_0$ with 3, 2, 4, 3, 5, 1, 5, 4, 1, respectively, and color $x_2x_3$ from $\{1,3\}$ with respect to 4 and $\sU_\phi(y_3)$.  This yields a good coloring of $G$.

So $\sU_\phi(y_6) = \{1,2,5\}$.  We then color $x_0x_1, x_1x_2, x_3x_4, x_4x_5, x_5x_6, x_6x_7, x_7x_0, x_4x_0$ with 4, 2, 1, 5, 3, 4, 1, 3, respectively, and color $x_2x_3$ from $\{3,4\}$ with respect to 1 and $\sU_\phi(y_3)$.  This yields a good coloring of $G$.

Up to relabeling colors and symmetry, this completes the proof of Subcase \ref{(4,5).5}, and so completes the proof of Case \ref{5.5.3}.  Thus, as we have exhausted all cases, the lemma holds.
\end{proof}

\begin{proof}[Proof of Lemma \ref{No5-6face}]
We assume $\psi$ and $\phi$ to be as described in the proof of Lemma \ref{No5-6face} in Section \ref{struct3}.  Thus, in order to extend $\phi$, it remains to color the edges of the cycle $x_0x_1x_2x_3x_4x_5$.  As a result, when we `color the cycle in order' we color the edges $x_0x_1, x_1x_2, x_2x_3, x_3x_4, x_4x_5, x_5x_0$ in this order.

We will break the following argument into two cases depending on $\phi(x_3y_3)$.  Within each case we consider $\phi(x_8)\setminus\{\alpha\}$, and within these subcases, we consider the values of $\alpha$ and $\beta$.  As in the proof of Lemma \ref{No5-5face}, we use claims to aid in the readability.

\setcounter{case}{0}
\begin{case}\label{(5,6).1}
$\phi(x_3y_3) = \phi(x_4y_4) = 1$.
\end{case}

\begin{subcase}
$\sU_\phi(x_8)\setminus\{\alpha\} = \{1,2\}$.
\end{subcase}

In this case, $\alpha, \beta \in \{3,4,5\}$.  So without loss of generality, assume $\alpha = 3$ and $\beta \neq 5$. Let $\obeta$ be such that $\{\beta, \obeta\} = \{3,4\}$.  

\begin{claim}
$4 \in \sU_\phi(y_1)$, and by symmetry, $\obeta \in \sU_\phi(y_4)$.
\end{claim}

\begin{proof}
Suppose $4 \notin \sU_\phi(y_1)$.  Let $\{\gamma_1, \gamma_2, \gamma_3\} = \{2,3,5\}$.  Suppose that $4 \notin \sU_\phi(y_3)$.  We color $x_0x_1$ with 4, and then color $x_1x_2$ from $\{\gamma_1, \gamma_2, \gamma_3\}\setminus\sU_\phi(y_2)$.  We may assume $x_1x_2$ is colored with $\gamma_1$.  We then color $x_3x_4$ with 4, and color $x_4x_5$ (and $x_5x_0$) from $\{2,5\}$ with respect to 4 and $\sU_\phi(y_4)$.  Let $\gamma$ denote the color used on $x_4x_5$, and color $\{\gamma_2,\gamma_3\}\setminus\{\gamma\}$.  This yields a good coloring of $G$.

So $4 \in \sU_\phi(y_3)$.  We color $x_0x_1$ and $x_1x_2$ with 4 and $\gamma_1$, respectively, where $\gamma_1 \notin \sU_\phi(y_2)$ as above.  Suppose $\gamma_1 = 5$.  Color $x_4x_5$ and $x_5x_0$ with 5 and 2, respectively.  We then color $x_3x_4$ (and $x_2x_3$) from $\{2,3\}$ with respect to 5 and $\sU_\phi(y_4)$.  This yields a good coloring of $G$.  A similar argument holds when $\gamma_1 = 2$ by switching the roles of 2 and 5.

So $\gamma_1 = 3$ and $\sU_\phi(y_2) = \{1,2,5\}$, otherwise we could recolor $x_1x_2$ with either 2 or 5 as above.  Now $\sU_\phi(y_4) = \{1,2,5\}$, otherwise color the cycle in order with 4, 2, 3, 5, 2, 5.  We then color the cycle in order with 5, 4, 2, 5, $\obeta$, 1.  These are good colorings of $G$ and prove the claim.
\end{proof}

\begin{claim}
$\sU_\phi(y_1) \cup \sU_\phi(y_2) = \{1,2,3,4,5\}$, and by symmetry, $\sU_\phi(y_3) \cup \sU_\phi(y_4) = \{1,2,3,4,5\}$.
\end{claim}

\begin{proof}
Suppose that $2 \notin \sU_\phi(y_1) \cup \sU_\phi(y_2)$.  Then $\sU_\phi(y_4) = \{1,2,4\}$, otherwise color $x_0x_1, x_1x_2$, $x_3x_4, x_4x_5, x_5x_0$ with 4, 2, 4, 2, 5, respectively, and color $x_2x_3$ from $\{3,5\}$ with respect to 4 and $\sU_\phi(y_3)$.  This yields a good coloring of $G$.   Now $\sU_\phi(y_3) = \{1,3,5\}$,  otherwise color the cycle in order with 4, 2, 3, 5, 2, 5.  We then color the cycle in order with 4, 2, 3, 4, 5, 2.  These are both good colorings of $G$.

A similar argument holds if $5 \notin \sU_\phi(y_1) \cup \sU_\phi(y_2)$ by switching the roles of 2 and 5.  So we may assume that only $3 \notin \sU_\phi(y_1) \cup \sU_\phi(y_2)$.  Suppose $\sU_\phi(y_3) \neq \{1,3,4\}$.  Color $x_0x_1, x_2x_3, x_3x_4$ with 4, 3, 4, respectively.  We then color $x_1x_2$ from $\{1,2\}\setminus\sU_\phi(y_1)$, and color $x_4x_5$ (and $x_5x_0$) from $\{2,5\}$ with respect to 4 and $\sU_\phi(y_4)$.  This yields a good coloring of $G$.

So $\sU_\phi(y_3) = \{1,3,4\}$.  Color $x_0x_1$ and $x_2x_3$ with 4 and 3, respectively. We then color $x_1x_2, x_4x_5$ (and $x_3x_4, x_5x_0$) from $\{2,5\}$ with respect to 4 and $\sU_\phi(y_1)$.  This yields a good coloring and proves the claim.
\end{proof}

\begin{claim}
$\sU_\phi(y_1) = \{1,3,4\}$ and $\sU_\phi(y_2) = \{1,2,5\}$, and by symmetry,  and $\sU_\phi(y_4) = \{1,\beta, \obeta\} = \{1,3,4\}$ and $\sU_\phi(y_3) = \{1,2,5\}$.
\end{claim}

\begin{proof}
Suppose $\sU_\phi(y_1) = \{1,4, 5\}$.  By the previous claim, $\sU_\phi(y_2) = \{1,2,3\}$.  Suppose $\beta \notin \sU_\phi(y_3)$.  Color $x_0x_1, x_1x_2, x_2x_3, x_4x_5, x_5x_0$ with 2, $\obeta$, $\beta$, $\obeta$, 5, respectively.  We then color $x_3x_4$ with a color from $\{2,5\}\setminus\sU_\phi(y_4)$.  This yields a good coloring of $G$.  

So $\beta \in \sU_\phi(y_3)$.  If $2 \notin \sU_\phi(y_4)$, then by the previous claim $\sU_\phi(y_3) = \{1,2, \beta\}$ and $\sU_\phi(y_4) = \{1, \obeta, 5\}$.  Then color the cycle in order with 2, 4, 5, 2, $\obeta$, 5.  This is a good coloring of $G$.  Thus, $\sU_\phi(y_4) = \{1,2,\obeta\}$, and so $\sU_\phi(y_3) = \{1,\beta, 5\}$.  If $\beta = 3$, we color the cycle in order with 2, 5, 3, 2, 5, 4, to obtain a good coloring of $G$.  If $\beta = 4$, we color the cycle in order with 5, 2, 4, 3, 5, 2, to obtain a good coloring of $G$.  

As above, a similar argument holds when $\sU_\phi(y_1) = \{1,2,4\}$ by switching the roles of 2 and 5.  Thus, as $1,4 \in \sU_\phi(y_1)$, the claim holds.
\end{proof}

Now color $x_0x_1, x_1x_2, x_2x_3, x_3x_4, x_5x_0$ with 5, 4, 3, 2, 2, respectively.  If $\beta = 3$, color $x_4x_5$ with 4.  If $\beta = 4$, color $x_4x_5$ with 5.  In either case, we obtain a good coloring of $G$, which proves the subcase.

\begin{subcase}
$\sU_\phi(x_8)\setminus\{\alpha\} = \{4,5\}$.
\end{subcase}

\begin{subsubcase}\label{1.1}
$(\alpha,\beta) = (1,1)$.
\end{subsubcase}

\begin{claim}
$\sU_\phi(y_1) = \{1,4,5\}$, and by symmetry $\sU_\phi(y_4) = \{1,4,5\}$.
\end{claim}

\begin{proof}
Suppose $\sU_\phi(y_1) \neq \{1,4,5\}$.  Also suppose $\sU_\phi(y_3) \neq \{1,2,3\}$.  Then $2 \in \sU_\phi(y_2)$, otherwise color  $x_2x_3, x_3x_4, x_5x_0$ with 2, 3, 2, respectively, and then color $x_4x_5, x_1x_2$ (and $x_0x_1$) from $\{4,5\}$ with respect to 3 and $\sU_\phi(y_4)$.  This yields a good coloring of $G$.  A similar argument holds if $3 \notin \sU_\phi(y_2)$ by switching the roles of 2 and 3.  So $\sU_\phi(y_2) = \{1,2,3\}$.  Now color $x_0x_1, x_1x_2, x_4x_5$ with 4, 5, 5, respectively, and then color $x_3x_4$ (and $x_2x_3, x_5x_0$) from $\{2,3\}$ with respect to 5 and $\sU_\phi(y_4)$.  This yields a good coloring of $G$.

So $\sU_\phi(y_3) = \{1,2,3\}$.  Suppose $2 \notin \sU_\phi(y_2)$.  Color $x_2x_3, x_4x_5, x_5x_0$ with 2, 3, 2, respectively.  We then color $x_3x_4, x_0x_1$ (and $x_1x_2$) from $\{4,5\}$ with respect to 3 and $\sU_\phi(y_4)$.  This yields a good coloring of $G$.

A similar argument holds if $3 \notin \sU_\phi(y_2)$ by switching the roles of 2 and 3.  So $\sU_\phi(y_2) = \{1,2,3\}$.  Now color $x_0x_1, x_1x_2, x_3x_4$ with 4, 5, 4, respectively.  We then color $x_4x_5$ (and $x_2x_3, x_5x_0$) from $\{2,3\}$ with respect to 4 and $\sU_\phi(y_4)$.  This yields a good coloring of $G$ and proves the claim.
\end{proof}

By the existence of $y_1y_2$ in $G'$, $\sU_\phi(y_2) \neq \{1,4,5\}$.  Color $x_1x_2, x_2x_3, x_4x_5$ with 4, 5, 4, respectively.  Then color $x_3x_4, x_0x_1$ (and $x_5x_0$) from $\{2,3\}$ with respect to 5 and $\sU_\phi(y_3)$.

\begin{subsubcase}
$(\alpha, \beta) = (1,3)$.
\end{subsubcase}

Color $x_0x_1$ and $x_5x_0$ with 3 and 2, respectively, and color $x_0x_1$ from $\{4,5\}$ with respect to 3 and $\sU_\phi(y_1)$.  Without loss of generality, assume $x_0x_1$ is colored with 4 so that $\sU_\phi(y_1) \neq \{1,3,4\}$.  Now color $x_2x_3$ from $\{2,5\}$ with respect to 4 and $\sU_\phi(y_2)$.  Let $\gamma$ denote the color used on $x_2x_3$.   We then color $x_3x_4$ from $\{2,3,5\}\setminus\{\gamma\}$ with respect to $\gamma$ and $\sU_\phi(y_3)$. Let $\delta$ denote the color used on $x_3x_4$.  If $\sU_\phi(y_4) \neq \{1,4,\delta\}$, then coloring $x_4x_5$ with 4 yields a good coloirng of $G$.  Note that $\gamma \neq \delta$.

So $\sU_\phi(y_4) = \{1,4,\delta\}$.  Assume $x_0x_1, x_1x_2, x_2x_3, x_5x_0$ are colored as above.  Suppose $\gamma = 5$.  If $5 \notin \sU_\phi(y_3)$, color $x_3x_4$ from $\{2,3\}\setminus\sU_\phi(y_4)$ and color $x_4x_5$ with 4.  This yields a good coloring of $G$.  

So $5 \in \sU_\phi(y_3)$.  Then $\sU_\phi(y_2) = \{1,2,4\}$, otherwise color the cycle in order with 3, 4, 2, 3, 5, 2.  This is a good coloring of $G$.  Since $\gamma = 5$, and $\gamma \neq \delta$, $\sU_\phi(y_4) \neq \{1,4,5\}$.  We then color the cycle in order with 4, 3, 2, 4, 5, 2.  This is a good coloring of $G$.

Thus $\gamma = 2$ so that $\sU_\phi(y_2) \neq \{1,2,4\}$ and $\sU_\phi(y_4) = \{1,\delta, 4\}\neq \{1,2,4\}$.  Now $\sU_\phi(y_3) = \{1,2,3\}$, otherwise color the cycle in order with 3, 4, 2, 3, 5, 2.  This is a good coloring of $G$.  We then color $x_2x_3, x_3x_4, x_4x_5, x_5x_0$ with 5, 2, 4, 2, respectively, and color $x_1x_2$ (and $x_0x_1$) from $\{3,4\}$ with respect to 5 and $\sU_\phi(y_2)$.  This yields a good coloring of $G$.

\begin{subsubcase}
$\alpha = \beta = 2$.
\end{subsubcase}

\begin{claim}
$\sU_\phi(y_1) = \{1,4,5\}$, and by symmetry $\sU_\phi(y_4) = \{1,4,5\}$.
\end{claim}

\begin{proof}
Suppose $\sU_\phi(y_1) \neq \{1,4,5\}$.  Also suppose $\sU_\phi(y_3) \neq \{1,2,3\}$.  Now $2 \in \sU_\phi(y_2)$, otherwise color $x_2x_3, x_3x_4, x_5x_0$ with 2,3,1, respectively, and color $x_1x_2, x_4x_5$ (and $x_0x_1$) from $\{4,5\}$ with respect to 2 and $\sU_\phi(y_2)$.  This yields a good coloring of $G$.  Also $3 \in \sU_\phi(y_2)$, otherwise color $x_2x_3, x_3x_4, x_5x_0$ with 3, 2, 3, respectively, and color $x_4x_5, x_1x_2$ (and $x_0x_1$) from $\{4,5\}$ with respect to 2 and $\sU_\phi(y_4)$.  This yields a good coloring of $G$.  

So $\sU_\phi(y_2) = \{1,2,3\}$.  Suppose $\sU_\phi(y_3) \neq \{1,2,3\}$.  We then color $x_2x_3$ and $x_3x_4$ with 2 and 3, respectively, and color $x_4x_5, x_1x_2$ (and $x_0x_1$) from $\{4,5\}$ with respect to 3 and $\sU_\phi(y_4)$.  This yields a good coloring of $G$.

So $\sU_\phi(y_3) = \{1,2,3\}$.  Suppose $\sU_\phi(y_4) \neq \{1,4,5\}$.  Now color $x_2x_3$ with 2. Then color $x_1x_2, x_4x_5$ (and $x_0x_1, x_3x_4$) from $\{4,5\}$ with respect to 2 and $\sU_\phi(y_2)$.  This yields a good coloring of $G$.

So $\sU_\phi(y_4) = \{1,4,5\}$.  Color $x_2x_3$ and $x_4x_5$ with 2 and 3, respectively.  We then color $x_1x_2$ (and $x_0x_1, x_3x_4$) from $\{4,5\}$ with respect to 2 and $\sU_\phi(y_2)$.  This yields a good coloring of $G$ and proves the claim.
\end{proof}

Suppose $\sU_\phi(y_3) \neq \{1,2,3\}$.  Color $x_0x_1, x_2x_3, x_3x_4, x_4x_5$ with 3, 2, 3, 4, respectively.  Then color $x_1x_2$ from $\{4,5\}$ with respect to 2 and $\sU_\phi(y_2)$.  This yields a good coloring of $G$.  

So $\sU_\phi(y_3) = \{1,2,3\}$, and by symmetry, $\sU_\phi(y_2) = \{1,2,3\}$.  We then color the cycle in order with 4, 3, 5, 4, 3, 1.  This is a good coloring of $G$.

\begin{subsubcase}
$(\alpha,\beta) = (2,3)$.
\end{subsubcase}

In the following, we will assume $x_5x_0$ is colored with 1.

\begin{claim}
$\sU_\phi(y_1) = \{1,4,5\}$, and by symmetry $\sU_\phi(y_4) = \{1,4,5\}$.
\end{claim}

\begin{proof}
Suppose $\sU_\phi(y_1) \neq \{1,4,5\}$.  Also suppose that $\sU_\phi(y_4) \neq \{1,4,5\}$.  Now $2 \in \sU_\phi(y_2)$, otherwise color $x_2x_3$ with 2, $x_1x_2$ (and $x_0x_1$) from $\{4,5\}$ with respect to 2 and $\sU_\phi(y_2)$, and color $x_3x_4$ (and $x_4x_5$) from $\{4,5\}$ with respect to 2 and $\sU_\phi(y_3)$.  This yields a good coloring of $G$.

So $2 \in \sU_\phi(y_2)$, and by a similar argument $3 \in \sU_\phi(y_2)$.  Thus, $\sU_\phi(y_2) = \{1,2,3\}$.  Since we are currently assuming that neither $\sU_\phi(y_1)$ nor $\sU_\phi(y_4)$ is $\{1,4,5\}$, by symmetry we deduce that $\sU_\phi(y_3) = \{1,2,3\}$.  We then color the cycle in order with 4, 5, 2, 4, 5, 1.  This is a good coloring of $G$.

So $\sU_\phi(y_4) = \{1,4,5\}$.  Suppose $\sU_\phi (y_3) \neq \{1,2,3\}$.  We color $x_2x_3, x_3x_4, x_4x_5$ with 3, 2, 4, respectively.  We then color $x_1x_2$ (and $x_0x_1$) from $\{4,5\}$ with respect to 3 and $\sU_\phi(y_2)$.  This yields a good coloring of $G$.

So $\sU_\phi(y_3) = \{1,2,3\}$.  Now color $x_2x_3$ and $x_4x_5$ with 3 and 2, respectively.  Then color $x_1x_2$ (and $x_0x_1, x_3x_4$) from $\{4,5\}$ with respect to 3 and $\sU_\phi(y_2)$.  This yields a good coloring of $G$ and proves the claim.
\end{proof}

Suppose $3 \notin \sU_\phi(y_2)$.  Color $x_1x_2$ and $x_3x_4$ with 3 and 2, respectively.  Then color $x_2x_3$ (and $x_0x_1, x_4x_5$) from $\{4,5\}$ with respect to 2 and $\sU_\phi(y_3)$.  This yields a good coloring of $G$.

So $3 \in \sU_\phi(y_2)$, and by symmetry, $2 \in \sU_\phi(y_3)$.  Color $x_0x_1$ and $x_3x_4$ with 3 and 2, respectively.  We then color $x_2x_3$ (and $x_1x_2, x_4x_5$) from $\{4,5\}$ with respect to 2 and $\sU_\phi(y_3)$.  This yields a good coloring of $G$.

Up to relabeling colors and symmetry, this completes all subcases and proves Case \ref{(5,6).1}

\begin{case}\label{(5,6).2}
$\phi(x_3y_3) = \phi(x_4y_4) = 2$.
\end{case}

\begin{subcase}
$\sU_\phi(x_8)\setminus\{\alpha\} = \{1,2\}$.
\end{subcase}

Without loss of generality, assume $\alpha = 3$ and $\beta, \obeta \in \{3,4\}$ such that $\{\beta, \obeta\} = \{3,4\}$.  

In the following, suppose we have colored $x_0x_1, x_4x_5, x_5x_0$ with 2,1, 5, respectively.  Let $\sigma$ denote this good partial coloring of $G$.  Let $\{\gamma_1, \gamma_2, \gamma_3\} = \{3,4,5\}$.  

\begin{claim}
$2 \in \sU_\sigma(y_1)$, and by symmetry $1 \in \sU_\phi(y_4)$.
\end{claim}

\begin{proof}
Suppose $2 \notin \sU_\sigma(y_1)$.  Color $x_1x_2$ with a color from $\{3,4,5\}\setminus\sU_\sigma(y_2)$.  Without loss of generaltiy, suppose it is $\gamma_1$.  If $\sU_\phi(y_3) \neq \{2,\gamma_2, \gamma_3\}$, color $x_3x_4$ (and $x_2x_3$) from $\{\gamma_2,\gamma_3\}$ with respect to 1 and $\sU_\phi(y_4)$.  This yields a good coloring of $G$.  

So $\sU_\sigma(y_3) = \{2,\gamma_2, \gamma_3\}$.  We then color $x_1x_2$ and $x_2x_3$ with $\gamma_2$ and $\gamma_1$, respectively, and color $x_3x_4$ from $\{\gamma_2, \gamma_3\}$ with respect to 1 and $\sU_\phi(y_4)$.  This yields a good coloring of $G$ and proves the claim.
\end{proof}

\begin{claim}
$\sU_\sigma(y_1) \cup \sU_\sigma(y_2) = \{1,2,3,4,5\}$, and by symmetry $\sU_\phi(y_3) \cup \sU_\phi(y_4) = \{1,2,3,4,5\}$.
\end{claim}

\begin{proof}
Without loss of generality, suppose $\gamma_1 \notin \sU_\sigma(y_1) \cup \sigma(y_2)$.  If $\sU_\sigma(y_3) \neq \{2,\gamma_2, \gamma_3\}$, color $x_1x_2$ with $\gamma_1$ and color $x_3x_4$ (and $x_2x_3$) from $\{\gamma_2, \gamma_3\}$ with respect to 1 and $\sU_\phi(y_4)$.  This yields a good coloring of $G$.

So $\sU_\phi(y_3) = \{2, \gamma_2, \gamma_3\}$.  We then color $x_2x_3$ with $\gamma_1$, color $x_1x_2$ from $\{\gamma_2, \gamma_3\}$ with respect to 2 and $\sU_\phi(y_1)$, and color $x_3x_4$ from $\{\gamma_2, \gamma_3\}$ with respect to 1 and $\sU_\phi(y_4)$.  This yields a good coloring of $G$ and proves the claim.
\end{proof}

Without loss of generality, assume $\sU_\sigma(y_1) = \{1,2,\gamma_1\}$ and $\sU_\sigma(y_2) = \{1,\gamma_2, \gamma_3\}$.  Suppose $\gamma_2 \notin \sU_\phi(y_4)$.  We then color $x_3x_4$ with $\gamma_2$ and color $x_2x_3$ (and $x_1x_2$) from $\{\gamma_1, \gamma_3\}$ with respect to $\gamma_2$ and $\sU_\phi(y_3)$.  This yields a good coloring of $G$.

So $\sU_\phi(y_4) = \{1,2,\gamma_2\}$, however a similar argument holds if $\gamma_3 \notin \sU_\phi(y_4)$.  This proves the subcase.

\begin{subcase}
$\sU_\phi(x_8)\setminus\{\alpha\} = \{1,5\}$.
\end{subcase}

\begin{subsubcase}
$\alpha = \beta = 2$.
\end{subsubcase}

\begin{claim}
$\sU_\phi(y_1) = \{1,3,5\}$.
\end{claim}

\begin{proof}
Suppose $\sU_\phi(y_1) \neq \{1,3,5\}$.  If $4 \notin \sU_\phi(y_3)$, color $x_2x_3, x_4x_5, x_5x_0$ with 4, 1, 4, respectively, color $x_1x_2$ (and $x_0x_1$) from $\{3,5\}$ with respect to 4 and $\sU_\phi(y_2)$, and color $x_3x_4$ from $\{3,5\}$ with respect to 1 and $\sU_\phi(y_4)$.  This yields a good coloring of $G$.

So $4 \in \sU_\phi(y_3)$, and by a similar argument $1 \in \sU_\phi(y_4)$.  Suppose $4 \notin \sU_\phi(y_2)$.  Color both $x_2x_3$ and $x_5x_0$ with 4.  We then color $x_3x_4, x_0x_1$ (and $x_1x_2, x_4x_5$) from $\{3,5\}$ with respect to 4 and $\sU_\phi(y_3)$.  This yields a good coloring of $G$.

So $4 \in \sU_\phi(y_2)$.  Suppose $3 \notin \sU_\phi(y_2)$.  Color $x_0x_1, x_2x_3, x_3x_4, x_4x_5, x_5x_0$ with 5, 3, 5, 4, 3, respectively, and color $x_1x_2$ from $\{2,4\}$ with respect to 5 and $\sU_\phi(y_1)$.  This yields a good coloring of $G$.  

Thus, $\sU_\phi(y_2) = \{1,3,4\}$, and by the existence of $y_1y_2$ in $G'$, $\sU_\phi(y_1) \neq \{1,3,4\}$.   Now $\sU_\phi(y_4) = \{1,2,3\}$, otherwise color the cycle in order with 3, 4, 5, 3, 1, 4.  We then color the cycle in order with 3, 4, 5, 1, 4, 5.  These are good colorings of $G$ and proves the claim.
\end{proof}

By the existence of $y_1y_2$ in $G'$, $\sU_\phi(y_2) \neq \{1,3,5\}$.  Suppose $\sU_\phi(y_3) \neq \{2,4,5\}$.  Color $x_0x_1, x_1x_2, x_2x_3, x_3x_4$ with 4, 3, 5, 4, respectively.  Then color $x_4x_5$ (and $x_5x_0$) from $\{1,3\}$ with respect to 4 and $\sU_\phi(y_4)$.  This yields a good coloring of $G$.

So $\sU_\phi(y_3) = \{2,4,5\}$.  By the existence of $y_3y_4$ in $G'$, $\sU_\phi(y_4) \neq \{2,4,5\}$.  So color the cycle in order with 4, 5, 3, 4, 5, 3.  This is a good coloring of $G$.

\begin{subsubcase}
$\alpha = \beta = 3$.
\end{subsubcase}

Suppose $2 \notin \sU_\phi(y_1)$.  Also, suppose $\sU_\phi(y_4) \neq \{1,2,5\}$.  Color $x_0x_1, x_3x_4, x_4x_5, x_5x_0$ with 2, 5, 1, 4, respectively.  Then color $x_2x_3$ (and $x_1x_2$) from $\{3,4\}$ with respect to 5 and $\sU_\phi(y_3)$.  This yields a good coloring of $G$.  So $\sU_\phi(y_4) = \{1,2,5\}$.  However, a similar argument holds if $\sU_\phi(y_4) \neq \{1,2,4\}$, a contradiction.

Thus, $2 \in \sU_\phi(y_1)$, and by symmetry $1 \in \sU_\phi(y_4)$.  Now suppose $3 \notin \sU_\phi(y_3)$.  Color $x_2x_3$ and $x_5x_0$ with 3 and 2, respectively.  Then color $x_1x_2, x_4x_5$ (and $x_0x_1, x_3x_4$) from $\{4,5\}$ with respect to 3 and $\sU_\phi(y_2)$.  This yields a good coloring of $G$.

So $3 \in \sU_\phi(y_3)$, and by symmetry $3 \in \sU_\phi(y_2)$.  Now suppose $\sU_\phi(y_1) \neq \{1,2,4\}$.  Then $\sU_\phi(y_4) = \{1,2,4\}$, otherwise color the cycle in order with 2, 4, 5, 4, 1, 4.  $\sU_\phi(y_3) = \{2,3,5\}$, otherwise color the cycle in order with 2, 4, 5, 3, 4, 5.  $\sU_\phi(y_1) = \{1,2,5\}$, otherwise color the cycle in order with 2, 5, 4, 3, 5, 4.  $\sU_\phi(y_2) = \{1,3,4\}$, otherwise color the cycle in order with 5, 3, 4, 5, 1, 4.  We then color the cycle in order with 4, 5, 3, 4, 5, 2.  These are each good colorings of $G$.

So $\sU_\phi(y_1) = \{1,2,4\}$.  Then $\sU_\phi(y_4) = \{1,2,5\}$, otherwise color the cycle in order with 2, 5, 4, 5, 1, 4.  $\sU_\phi(y_3) = \{2,3,4\}$, otherwise color the cycle in order with 2, 5, 4, 3, 5, 4.  $\sU_\phi(y_2) = \{1,3,5\}$, otherwise color the cycle in order with 4, 3, 5, 4, 1, 2.  We then color the cycle in order with 5, 4, 3, 5, 4, 2.  These are all good colorings of $G$.

\begin{subsubcase}
$(\alpha,\beta) = (3,2)$.
\end{subsubcase}

\begin{claim}
$\sU_\phi(y_1) = \{1,2,5\}$.
\end{claim}

\begin{proof}
Suppose $\sU_\phi(y_1) \neq \{1,2,5\}$.  Color $x_2x_3, x_4x_5, x_5x_0$ with 4, 1, 4, respectively.  Then color $x_1x_2$ (and $x_0x_1$) from $\{2,5\}$ with respect to 4 and $\sU_\phi(y_2)$.  If $4 \notin \sU_\phi(y_3)$, color $x_3x_4$ from $\{3,5\}$ with respect to 1 and $\sU_\phi(y_4)$.  Similarly, if $1 \notin \sU_\phi(y_4)$, color $x_3x_4$ from $\{3,5\}$ with respect to 4 and $\sU_\phi(y_3)$.  These are good colorings of $G$.

So $4 \in \sU_\phi(y_3)$ and $1 \in \sU_\phi(y_4)$.  Suppose $\sU_\phi(y_2) \neq \{1,2,4\}$.  Color $x_0x_1, x_1x_2, x_2x_3, x_5x_0$ with 5, 2, 4, 4, respectively.  Then color $x_3x_4$ (and $x_4x_5$) from $\{3,5\}$ with respect to 4 and $\sU_\phi(y_3)$.  This yields a good coloring of $G$.

So $\sU_\phi(y_2) = \{1,2,4\}$, and by the existence of $y_1y_2$ in $G'$, $\sU_\phi(y_1) \neq \{1,2,4\}$.  We then color the cycle in order with 4, 2, 5, 3, 4, 5.  This is a good coloring of $G$ and proves the claim.
\end{proof}

\begin{claim}
$\sU_\phi(y_2) = \{1,3,4\}$.
\end{claim}

\begin{proof}
Suppose $\sU_\phi(y_2) \neq \{1,3,4\}$.  Also suppose $\sU_\phi(y_4) \neq \{1,2,5\}$.  Color $x_0x_1, x_3x_4, x_4x_5, x_5x_0$ with 2, 5, 1, 4, respectively, and color $x_2x_3$ (and $x_1x_2$) from $\{3,4\}$ with respect to 5 and $\sU_\phi(y_3)$.  This yields a good coloring of $G$.

So $\sU_\phi(y_4) = \{1,2,5\}$.  Suppose $\sU_\phi(y_3) \neq \{2,3,5\}$.  Then color $x_0x_1, x_2x_3, x_3x_4, x_4x_5, x_5x_0$ with 2, 5, 3, 1, 4, respectively, and color $x_1x_2$ from $\{3,4\}$ with respect to 5 and $\sU_\phi(y_2)$.  This yields a good coloring of $G$.

So $\sU_\phi(y_3) = \{2,3,5\}$.  We color the cycle in order with 2, 3, 4, 3, 5, 4.  This is a good coloring of $G$ and proves the claim.
\end{proof}

Now $\sU_\phi(y_3) = \{1,2,5\}$, otherwise color $x_0x_1, x_1x_2, x_2x_3, x_3x_4, x_5x_0$ with 4, 3, 5, 1, 5, respectively, and color $x_4x_5$ from $\{3,5\}$ with respect to 1 and $\sU_\phi(y_4)$.  This is a good coloring of $G$.  We then color $x_0x_1, x_1x_2, x_2x_3, x_4x_5$ with 4, 5, 3, 4, respectively, and color $x_3x_4$ (and $x_5x_0$) from $\{1,5\}$ with respect to 4 and $\sU_\phi(y_4)$.  This yields a good coloring of $G$.

\begin{subsubcase}
$(\alpha,\beta) = (2,3)$.
\end{subsubcase}

\begin{claim}
$\sU_\phi(y_4) = \{1,2,5\}$.
\end{claim}

\begin{proof}
Suppose $\sU_\phi(y_4) \neq \{1,2,5\}$.  Also suppose  $\sU_\phi(y_2) \neq \{1,3,4\}$.  Color $x_0x_1$ and $x_5x_0$ with 5 and 4, respectively.  Then color $x_1x_2$ (and $x_2x_3$) from $\{3,4\}$ with respect to 5 and $\sU_\phi(y_1)$.  Let $\gamma$ denote the color used on $x_1x_2$.  We then color $x_3x_4$ (and $x_4x_5$) from $\{1,5\}$ with respect to $\gamma$ and $\sU_\phi(y_3)$.  This yields a good coloring of $G$.

So $\sU_\phi(y_2) = \{1,3,4\}$.  By the existence of $y_1y_2$ in $G'$, $\sU_\phi(y_1) \neq \{1,3,4\}$.  Now color $x_0x_1, x_1x_2, x_2x_3, x_4x_5, x_5x_0$ with 3, 4, 5, 1, 4, respectively.  If $5 \notin \sU_\phi(y_3)$, color $x_3x_4$ from $\{3,4\}$ with respect to 1 and $\sU_\phi(y_4)$.  This yields a good coloring of $G$.

So $5 \in \sU_\phi(y_3)$, and by a similar argument $1 \in \sU_\phi(y_4)$.   Now $\sU_\phi(y_1) = \{1,3,5\}$, otherwise color the cycle in order with 3, 5, 4, 3, 5, 4.  We then color the cycle in order with 3, 2, 4, 3, 5, 4.  These are good colorings of $G$ and prove the claim.
\end{proof}

\begin{claim}
$\sU_\phi(y_2) = \{1,4,5\}$.
\end{claim}

\begin{proof}
Suppose $\sU_\phi(y_2) \neq \{1,4,5\}$.  By the existence of $y_3y_4$ in $G'$ and the previous claim, $\sU_\phi(y_3) \neq \{1,2,5\}$.  Then $\sU_\phi(y_1) = \{1,3,4\}$, otherwise color the cycle in order with 3, 4, 5, 1, 4, 5.  $\sU_\phi(y_3) = \{2,3,4\}$, otherwise color the cycle in order with 3, 5, 4, 3, 5, 4.  $\sU_\phi(y_2) = \{1,3,5\}$, otherwise color the cycle in order with 4, 5, 3, 5, 4, 5.  We then color the cycle in order with 4, 2, 3, 5, 4, 5.  These are each good colorings of $G$ and prove the claim.
\end{proof}

Again, by the existence of $y_3y_4$ in $G'$, $\sU_\phi(y_3) \neq \{1,2,5\}$.  So $\sU_\phi(y_1) = \{1,3,4\}$, otherwise color the cycle in order with 4, 3, 5, 1, 4, 5.  Also $\sU_\phi(y_3) = \{2,3,4\}$, otherwise color the cycle in order with 5, 3, 4, 3, 1, 4.  We then color the cycle in order with 4, 5, 3, 5, 4, 1.  These are good colorings of $G$.

\begin{subsubcase}
$(\alpha,\beta) = (3,4)$.
\end{subsubcase}

\begin{claim}
$\sU_\phi(y_4) = \{1,2,5\}$.
\end{claim}

\begin{proof}
Suppose $\sU_\phi(y_4) \neq \{1,2,5\}$.  Also suppose $\sU_\phi(y_1) \neq \{1,4,5\}$.  Color $x_2x_3$ and $x_5x_0$ with 3 and 2, respectively.  Then color $x_1x_2$ (and $x_0x_1$) from $\{4,5\}$ with respect to 3 and $\sU_\phi(y_2)$, and color $x_3x_4$ (and $x_4x_5$) from $\{1,5\}$ with respect to 3 and $\sU_\phi(y_3)$.  This yields a good coloring of $G$.

So  $\sU_\phi(y_1) = \{1,4,5\}$.  Now $\sU_\phi(y_2) = \{1,3,4\}$, otherwise color $x_0x_1, x_1x_2, x_2x_3, x_5x_0$ with 5, 3, 4, 2, respectively, and color $x_3x_4$ (and $x_4x_5$) from $\{1,5\}$ with respect to 4 and $\sU_\phi(y_3)$.  This yields a good coloring of $G$.

Also $\sU_\phi(y_3) = \{2,3,5\}$, otherwise color the cycle in order with 4, 2, 3, 5, 1, 2.  We then color the cycle in order with 5, 2, 4, 3, 5, 2.  These are good colorings of $G$ and prove the claim.
\end{proof}

By the existence of $y_3y_4$ in $G'$, $\sU_\phi(y_3) \neq \{1,2,5\}$.

\begin{claim}
$\sU_\phi(y_3) = \{2,4,5\}$.
\end{claim}

\begin{proof}
Suppose $\sU_\phi(y_3) \neq \{2,4,5\}$.  Color $x_0x_1, x_2x_3, x_3x_4, x_4x_5, x_5x_0$ with 5, 4, 5, 3, 2, respectively.  If $5 \notin \sU_\phi(y_1)$, color $x_1x_2$ from $\{2,3\}$ with respect to 4 and $\sU_\phi(y_2)$.  This yields a good coloring of $G$.

So $5 \in \sU_\phi(y_1)$, and by a similar argument $4 \in \sU_\phi(y_2)$.  We then color the cycle in order with 2, 3, 5, 1, 3, 5.  This is a good coloring of $G$ and proves the claim.
\end{proof}

Now $\sU_\phi(y_2) = \{1,4,5\}$, otherwise color $x_0x_1, x_3x_4, x_4x_5, x_5x_0$ with 2, 1, 3, 5, respectively, and then color $x_1x_2$ (and $x_2x_3$) from $\{4,5\}$ with respect to 2 and $\sU_\phi(y_1)$.  This yields a good coloring of $G$.  We then color 
 $x_0x_1, x_2x_3, x_3x_4, x_4x_5, x_5x_0$ with 5, 4, 3, 5, 2, respectively, and color $x_1x_2$ from $\{2,3\}$ with respect to 5 and $\sU_\phi(y_1)$.  This yields a good coloring of $G$.

Up to relabeling colors and symmetry, this proves the subcase.

\begin{subcase}
$\sU_\phi(x_8)\setminus\{\alpha\} = \{4,5\}$.
\end{subcase}

\begin{subsubcase}
$\alpha = 1$ and $\beta \in \{1,2\}$.
\end{subsubcase}

\begin{claim}
$\sU_\phi(y_4) = \{2,4,5\}$.
\end{claim}

\begin{proof}
Suppose $\sU_\phi(y_4) \neq \{2,4,5\}$.  Color $x_0x_1, x_2x_3, x_5x_0$ with 2, 3, 3, respectively, and color $x_3x_4$ (and $x_4x_5$) from $\{4,5\}$ with respect to 3 and $\sU_\phi(y_3)$.  If $2 \notin \sU_\phi(y_1)$, color $x_1x_2$ from $\{4,5\}$ with respect to 3 and $\sU_\phi(y_2)$.  This yields a good coloring of $G$.

So $2 \in \sU_\phi(y_1)$, and by symmetry $3 \in \sU_\phi(y_2)$.  Suppose $3 \notin \sU_\phi(y_3)$.  Color both $x_2x_3$ and $x_5x_0$ with 3. We then color $x_1x_2, x_4x_5$ (and $x_0x_1, x_3x_4$) from $\{4,5\}$ with respect to 3 and $\sU_\phi(y_2)$.  This yields a good coloring of $G$.

So $3 \in \sU_\phi(y_3)$.  Suppose $1 \notin \sU_\phi(y_4)$ or $\{4,5\} \cap \sU_\phi(y_4)= \emptyset$.  Color $x_0x_1, x_3x_4, x_5x_0$ with 2, 1, 3, respectively.  Then color $x_1x_2, x_4x_5$ (and $x_2x_3$) from $\{4,5\}$ with respect to 2 and $\sU_\phi(y_1)$.  This yields a good coloring of $G$.

So without loss of generality, $\sU_\phi(y_4) = \{1,2,4\}$.  Now $\sU_\phi(y_1) = \{1,2,5\}$, otherwise color the cycle in order with 2, 5, 4, 1, 5, 3.  We then color $x_0x_1, x_1x_2, x_4x_5$ with 2, 3, 3, respectively, and color $x_2x_3, x_5x_0$ (and $x_3x_4$) from $\{4,5\}$ with respect to 3 and $\sU_\phi(y_2)$.   This yields a good coloring of $G$ and proves the claim.
\end{proof}

Recall that by the existence of $y_3y_4$ in $G'$, $\sU_\phi(y_3) \neq \{2,4,5\}$.  Suppose $\sU_\phi(y_2) \neq \{1,4,5\}$.  Color $x_0x_1$ and $x_4x_5$ with 2 and 3, respectively.  Then color $x_1x_2, x_3x_4$ (and $x_2x_3, x_5x_0$) from $\{4,5\}$ with respect to 2 and $\sU_\phi(y_1)$.  This yields a good coloring of $G$.

So $\sU_\phi(y_2) = \{1,4,5\}$, and by the existence of $y_1y_2$ in $G'$, $\sU_\phi(y_1) \neq \{1,4,5\}$.  $\sU_\phi(y_3) = \{1,2,3\}$, otherwise color the cycle in order with 4, 5, 3, 1, 5, 3.    This is a good coloring of $G$.

If $\beta = 2$, color the cycle in order with 4, 5, 3, 4, 1, 3.  If $\beta = 1$, color $x_1x_2, x_4x_5, x_5x_0$ with 3, 3, 2, respectively, and color $x_0x_1, x_3x_4$ (and $x_2x_3$) from $\{4,5\}$ with respect to 3 and $\sU_\phi(y_1)$.  This yields a good coloring of $G$.

\begin{subsubcase}
$(\alpha, \beta) = (2,1)$.
\end{subsubcase}

\begin{claim}
$\sU_\phi(y_1) = \{1,4,5\}$, and by symmetry $\sU_\phi(y_4) = \{2,4,5\}$.
\end{claim}

\begin{proof}
Suppose $\sU_\phi(y_1) \neq \{1,4,5\}$.  Also suppose $\sU_\phi(y_4) \neq \{2,4,5\}$.  Then color both $x_2x_3$ and $x_5x_0$ with 3, color $x_1x_2$ (and $x_0x_1$) from $\{4,5\}$ with respect to 3 and $\sU_\phi(y_2)$, and color $x_3x_4$ (and $x_4x_5$) from $\{4,5\}$ with respect to 3 and $\sU_\phi(y_3)$.  This yields a good coloring of $G$.

So $\sU_\phi(y_4) = \{2,4,5\}$.  Suppose $\sU_\phi(y_2) \neq \{1,4,5\}$.  We then color $x_0x_1, x_4x_5, x_5x_0$ with 3, 4, 3, respectively, and color $x_1x_2, x_3x_4$ (and $x_2x_3$) from $\{4,5\}$ with respect to 3 and $\sU_\phi(y_1)$.  This yields a good coloring of $G$.

So $\sU_\phi(y_2) = \{1,4,5\}$.  Also $\sU_\phi(y_1) = \{1,2,3\}$, otherwise color the cycle in order with 3, 2, 4, 5, 3, 5.  $\sU_\phi(y_3) = \{1,2,3\}$, otherwise color the cycle in order with 4, 5, 3, 1, 4, 3.  We then color the cycle in order with 4, 2, 5, 3, 4, 3.  These are good colorings of $G$ and prove the claim.
\end{proof}

Now color $x_1x_2, x_3x_4, x_5x_0$ with 3, 1, 3, respectively.  If $3 \notin \sU_\phi(y_2)$, color $x_2x_3$ (and $x_0x_1, x_4x_5$) from $\{4,5\}$ with respect to 1 and $\sU_\phi(y_3)$.  This yields a good coloring of $G$.

So $3 \in \sU_\phi(y_2)$, and by a similar argument $1 \in \sU_\phi(y_3)$.  Similar arguments show that $2 \in \sU_\phi(y_2)$ and $3 \in \sU_\phi(y_3)$ by coloring $x_1x_2, x_3x_4, x_5x_0$ with 2, 3, 3, respectively.  

So $\sU_\phi(y_2) = \sU_\phi(y_3) = \{1,2,3\}$.  We then color the cycle in order with 3, 5, 4, 5, 3, 4.  This is a good coloring of $G$.

\begin{subsubcase}
$\alpha = \beta = 3$.
\end{subsubcase}

\begin{claim}
$\sU_\phi(y_1) = \{1,4,5\}$, and by symmetry $\sU_\phi(y_4) = \{2,4,5\}$.
\end{claim}

\begin{proof}
Suppose $\sU_\phi(y_4) \neq \{2,4,5\}$.  Color $x_0x_1, x_2x_3, x_5x_0$ with 2, 3, 1, respectively, and color $x_3x_4$ (and $x_4x_5$) from $\{4,5\}$ with respect to 3 and $\sU_\phi(y_3)$.  If $2 \notin \sU_\phi(y_1)$, color $x_1x_2$ from $\{4,5\}$ with respect to 3 and $\sU_\phi(y_2)$.  This yields a good coloring of $G$.

So $3 \in \sU_\phi(y_2)$ and by a similar argument $2 \in \sU_\phi(y_1)$.  Suppose $3 \notin \sU_\phi(y_3)$.  Color $x_2x_3$ and $x_5x_0$ with 3 and 1, respectively.  Then color $x_1x_2, x_4x_5$ (and $x_0x_1, x_3x_4$) from $\{4,5\}$ with respect to 3 and $\sU_\phi(y_2)$.  This yields a good coloring of $G$.

So $3 \in \sU_\phi(y_3)$.  Suppose $1 \notin \sU_\phi(y_4)$ or $\{4,5\} \cap \sU_\phi(y_4)= \emptyset$.  Color $x_0x_1, x_3x_4, x_5x_0$ with 2, 1, 1, respectively.  Then color $x_1x_2, x_4x_5$ (and $x_2x_3$) from $\{4,5\}$ with respect to 2 and $\sU_\phi(y_1)$.  This yields a good coloring of $G$.

So without loss of generality, $\sU_\phi(y_4) = \{1,2,4\}$.  Then $\sU_\phi(y_1) = \{1,2,5\}$, otherwise color the cycle in order with 2, 5, 4, 1, 5, 1.  $\sU_\phi(y_3) = \{2,3,5\}$, otherwise color the cycle in order with 2, 4, 5, 3, 1, 4.  $\sU_\phi(y_2) = \{1,3,4\}$, otherwise color the cycle in order with 2, 3, 4, 5, 1, 4.  We then color the cycle in order with 4, 5, 3, 4, 5, 1.  These are each good colorings of $G$ and prove the claim.
\end{proof}

Recall that by the existence of $y_1y_2$ and $y_3y_4$ in $G'$, $\sU_\phi(y_3) \neq \{2,4,5\}$ and $\sU_\phi(y_3) \neq \{2,4,5\}$.   Thus, we color the cycle in order with 2, 4, 5, 4, 1, 4.  This is a good coloring of $G$.

\begin{subsubcase}
$\alpha = 3$ and $\beta \in \{1,2\}$.
\end{subsubcase}

Let $\obeta \in \{1,2\}$ such that $\{\beta, \obeta\} = \{1,2\}$.

\begin{claim}
$\sU_\phi(y_4) = \{2,4,5\}$.
\end{claim}

\begin{proof}
Suppose $\sU_\phi(y_4) \neq \{2,4,5\}$.   Also suppose $\sU_\phi(y_1) \neq \{1,4,5\}$.  Color $x_2x_3$ and $x_5x_0$ with 3 and $\obeta$, respectively.  Then color $x_1x_2$ (and $x_0x_1$) from $\{4,5\}$ with respect to 3 and $\sU_\phi(y_2)$, and color $x_3x_4$ (and $x_4x_5$) from $\{4,5\}$ with respect to 3 and $\sU_\phi(y_3)$.  This yields a good coloring of $G$.

So $\sU_\phi(y_1) = \{1,4,5\}$.  Now suppose $\sU_\phi(y_2) \neq \{1,2,3\}$.  Color $x_0x_1, x_1x_2, x_2x_3, x_5x_0$ with 5, 2, 3, $\obeta$, respectively.  Then color $x_3x_4$ (and $x_4x_5$) from $\{4,5\}$ with respect to 3 and $\sU_\phi(y_3)$.  This yields a good coloring of $G$.

So $\sU_\phi(y_2) = \{1,2,3\}$.  Now color $x_0x_1, x_1x_2, x_3x_4, x_5x_0$ with 5, 2, 3, $\obeta$, respectively.  If $3 \notin \sU_\phi(y_3)$, color $x_4x_5$ (and $x_2x_3$) from $\{4,5\}$ with respect to 3 and $\sU_\phi(y_4)$.  This yields a good coloring of $G$.

So $3 \in \sU_\phi(y_3)$, and by a similar argument $3 \in \sU_\phi(y_4)$.  Now color $x_1x_2, x_4x_5, x_5x_0$ with 2, 3, $\obeta$, respectively.  Then color $x_3x_4, x_0x_1$ (and $x_2x_3$) from $\{4,5\}$ with respect to 3 and $\sU_\phi(y_4)$.  This yields a good coloring of $G$ and proves the claim.
\end{proof}

Recall that by the existence of $y_3y_4$ in $G'$, $\sU_\phi(y_3) \neq \{2,4,5\}$.   Color $x_1x_2, x_4x_5,x_5x_0$ with 2,3, $\obeta$, respectively.  If $2 \notin \sU_\phi(y_1)$, color $x_2x_3$ (and $x_0x_1, x_3x_4$) from $\{4,5\}$ with respect to 2 and $\sU_\phi(y_2)$.  This yields a good coloring of $G$.

So $2 \in \sU_\phi(y_1)$, and by a similar argument $2 \in \sU_\phi(y_2)$.  We then color $x_0x_1, x_2x_3, x_3x_4, x_5x_0$ with 2, 4, 5, 4, respectively, and color $x_1x_2$ from $\{3,5\}$ with respect to $\sU_\phi(y_1)$.  If $\beta = 1$, color $x_4x_5$ with 3.  If $\beta = 2$, color $x_4x_5$ with 1.  In either case we obtain a good coloring of $G$.

Up to symmetry and permuting colors, this completes the subcase, and so completes the proof of Case \ref{(5,6).2}.  As we have exhausted all cases, the lemma holds.
\end{proof}

\end{document}